\documentclass[letterpaper,12pt,reqno]{amsart}

\usepackage[table]{xcolor}
\usepackage{txfonts}

\usepackage{amsfonts}
\usepackage{multirow}
\usepackage{amssymb}
\usepackage{amsmath}
\usepackage{amsthm}
\usepackage{amscd}
\usepackage{verbatim}
\usepackage{subfigure}	
\usepackage{cancel}	
\usepackage[small,bf]{caption}
\usepackage{graphicx}
\usepackage{enumerate}
\makeatletter
\renewcommand*\env@matrix[1][c]{\hskip -\arraycolsep
  \let\@ifnextchar\new@ifnextchar
  \array{*\c@MaxMatrixCols #1}}
\makeatother
\usepackage{tikz}

\newtheorem{theorem}{Theorem}[section] 
\newtheorem{lemma}[theorem]{Lemma}

\newtheorem{corollary}[theorem]{Corollary}

\newtheorem{ex}[theorem]{Example}
\newtheorem{question}[theorem]{Question}

\theoremstyle{definition}
\newtheorem{example}[theorem]{Example}

\theoremstyle{definition}

\theoremstyle{definition}

\theoremstyle{plain}

\theoremstyle{remark}

\newtheorem{intthm}{Theorem}

\newcommand{\tcr}[1]{\textcolor{red}{#1}}
\newcommand{\tcb}[1]{\textcolor{blue}{#1}}
\newcommand{\tcc}[1]{\textcolor{cyan}{#1}}
\newcommand{\cost}[1]{\text{cost}\left( #1\right )}
\newcommand{\diam}[1]{\text{diam}\left( #1\right )}
\newcommand{\rad}[1]{\text{rad}\left( #1\right )}

\newcommand{\TG}[2]{C_{#1}\Box C_{#2} }
\newcommand{\grid}[2]{G_{#1,#2} }

\newcommand{\floor}[1]{\left \lfloor #1 \right\rfloor}

\input xy	
\xyoption{all}

\begin{document}

\author{Erik Insko}

\address{Erik Insko, Department of Mathematics\\
Florida Gulf Coast University\\
10501 FGCU Blvd. South\\
Fort Myers, FL 33965-6565}

\email{}

\urladdr{}

\author{Bethany Kubik}

\address{Bethany Kubik, Department of Mathematics and Statistics\\
University of Minnesota Duluth\\
140 Solon Campus Center\\
1117 University Drive\\
Duluth, MN 55812-3000}

\email{}

\urladdr{}

\author{Candice Price}

\address{Candice Price, Department of Mathematics\\
University of San Diego\\
Serra Hall 133\\
5998 Alcal\'{a} Park\\
San Diego, CA 92110}

\email{}

\urladdr{}

\title{Upper Broadcast Domination of Toroidal Grids and a 
Classification of Diametrical Trees}



\begin{abstract}
A \emph{broadcast} on a graph $G=(V,E)$ is a 
function $f:V \rightarrow \{0,1, \ldots, \diam{G}\}$ satisfying $f(v) \leq e(v)$ for all $v \in V$,
where $e(v)$ denotes the eccentricity of $v$ and $\diam{G}$ 
denotes the diameter of $G$.
We say that a broadcast dominates $G$ if every vertex can hear at least one broadcasting node.
The upper domination number is the maximum cost of all 
possible minimal broadcasts, 
where the cost of a broadcast is defined as $\cost{f}= \sum_{v \in V}f(v)$.
In this paper we establish both the 
upper domination number and the upper broadcast domination 
number on toroidal grids.  
In 
addition, we classify all diametrical trees, that is, trees whose upper domination number is equal to its diameter.
\end{abstract}

\maketitle

\section{Introduction}

The study of domination theory began in 1958 
with Berge's book~\cite{berge} which introduced the 
\emph{coefficient of external 
stability}, later renamed the \emph{domination number}.  
More than 80 domination related parameters 
have been defined and studied on graphs since then. 
In 1968, Liu discussed the concept of dominance in 
communication networks where the nodes represent cities with broadcast stations, and 
two cities (nodes) are connected by an edge
if they can hear each other's broadcasts. In this instance, a dominating set is a collection of cities whose 
broadcasts reach every city in the network \cite[
Example 9.1]{Liu}.  
In his 2004 PhD thesis, Erwin generalized this concept of domination to model the situation where the cities may build 
broadcast stations that can broadcast across multiple edges, but where the cost of building a stronger
broadcast station is proportional to the strength of the broadcast \cite{erwin}. 
In this model, a broadcast on a graph $G=(V,E)$ is a function $f~\colon~V \rightarrow \{0,1, \ldots, \diam{G}\}$ satisfying $f(v) \leq e(v)$ for all $v \in V$,
where $e(v)$ denotes the eccentricity of $v$ and $\diam{G}$ 
denotes the diameter of $G$.  The \emph{cost} of a broadcast $f$ is the sum $\cost{f}= \sum_{v \in V}f(v)$, 
and the lowest cost of any broadcast on a graph $G$
is called the broadcast domination number $\gamma_b(G)$:
\[ \gamma_b(G) = \min \{\cost{f} : f \text{ is a dominating broadcast of } G \} .\]
In 2005, Dunbar, Erwin, Haynes, Hedetniemi, and Hedetniemi noted the similarity with Liu's model and extended the study of dominating broadcasts in graphs \cite{DEHHH05}.
Extensive research has resulted in the area of broadcast domination and its variants 
\cite{BouZem14,bresar,brewster,heggernes,herke.rt,koh,soh,Mynhardt}.

Erwin also defined the upper broadcast domination number $\Gamma_b(G)$. This is the maximum cost 
of any minimal broadcast:
\[ \Gamma_b(G) = \max \{\cost{f} : f \text{ is a minimal dominating broadcast of } G \} .\]
The benefit of finding a $\Gamma_b$ dominating set is it ensures most of 
the graph can hear more than one broadcast tower, and 
yet the expense of each
tower is justified because the broadcast is minimal, i.e., there is at least one vertex per broadcasting tower that hears 
only that
tower \cite[Theorem 3]{DEHHH05}.   

Restricting the strength of all broadcasts to $f: V \rightarrow \{0,1\}$, we recover the (regular) domination number $\gamma$ and
upper domination number $\Gamma$.
Domination of grids $P_n \Box P_m$ and toroidal grids $C_m\Box C_n$
has been the focus of a considerable amount of literature over the past 30 years
\cite{AlaCreIsoOstPet11,Chang92,CocHarHedWim85,EZKN07,GonPinRaoTho11,GraMol97,JacKin86,KS95,Sha12}.
More recently, the broadcast domination of products of paths and cycles 
has also received a great deal of attention \cite{blessing,BouZem14, bresar,DEHHH05,koh,soh}.
In particular in 2014, Bre\v{s}ar and \v{S}pacapan studied broadcast domination in graph products. Their work 
showed that $\gamma_b(C_m\Box C_n)=\rad{C_m\Box C_n}-1$ if and only if $m$ and $n$ are both even and $\gamma_b(C_m\Box 
C_n)=\rad{C_m\Box C_n}$ otherwise \cite[Theorem 4.6]{bresar}.
In a related paper, Koh and Soh proved that $\gamma_b(C_m\Box C_n)=\lceil \frac{m+n}{2}\rceil-1$ as their main result
\cite[Theorem 1.3]{soh}.  
Following in this tradition, in Section \ref{tori} we consider 
the invariants $\Gamma(C_m\Box C_n)$ and $\Gamma_b(C_m\Box C_n)$.  
Our first main result is Theorem~\ref{theorem:GammaTG}.
\begin{intthm}
The upper domination number of $C_m\Box C_n$ is \[ \Gamma( C_m\Box C_n) = \begin{cases}  \frac{m 
\cdot 
n}{2} & \text{ if } m,n \text{ even } \\ 
\frac{m \cdot(n - 1)}{2} & \text{ if } m \text{ even}, n \text{ odd } \\
\frac{(m-1) \cdot n }{2} & \text{ if } m \text{ odd},  n \text{ even } \\
\frac{(m-1) \cdot (n-1) }{2} + 1 & \text{ if } m,n \text{ odd. } \\

\end{cases} \]  

\end{intthm}

The second main result is proved in Theorem~\ref{theorem:TGIF}.
\begin{intthm} 
 For any $3 \leq m \leq n$, the upper broadcast domination number of $C_m\Box C_n$ is
 \[ \Gamma_b ( C_m\Box C_n) = m \cdot \Gamma_b(C_n) \]
 where $\Gamma_b(C_n)$ is equal to $n-2$ is $n$ is even and $n-3$ if $n$ is odd.
\end{intthm}

Motivated by an open question posed by Dunbar et al., Herke and Mynhardt studied graphs $G$ with 
$\gamma_b(G) = \rad{G}$ and called these graphs \emph{radial graphs}.  
They classified all radial trees as trees 
whose diametrical path can be split into two even length pieces by removing a path consisting 
of vertices of degree 2 \cite[Theorem 1]{herke.rt}.  
In light of this work, it seems natural to consider \emph{diametrical graphs}, i.e., graphs $G$
whose upper broadcast domination number equals the 
diameter, $\Gamma_b(G) = \diam{G}$.  
Interestingly enough, while this manuscript was in preparation, 
a preprint by Mynhardt and Roux appeared on the arXiv that posed the classification 
of all diametrical trees as an open problem \cite[Problem 7]{Mynhardt}. 
In Section \ref{diametrical} we
classify all diametrical trees by proving that diametrical trees are a subfamily of 
lobster graphs.  The main theorem of Section \ref{diametrical} is the following result:

\begin{intthm} 
 A tree $T$ is diametrical if and only if it is a lobster graph containing only limbs of types $A$, $B$, and $C$ depicted in 
Figure \ref{fig:3limbs1} such that the number of limbs is less than half the diameter of the graph and the distance between each pair of 
adjacent limbs or an endpoint $e_i$ satisfies the following inequalities.
\begin{center}
\begin{tabular}{| c|c| c| } \hline 
$d(A,A) \geq 4$ & $d(A,B) \geq 3$ & 
$d(A,C) \geq 3$ \\ $d(B,B) \geq 3$ &
$d(B,C) \geq 2$ & 
$d(C,C) \geq 2$  \\ 
$d(e_i,A) \geq 2$ & $d(e_i,B) \geq 2$ & $d(e_i,C) \geq 1$  \\ \hline
\end{tabular}
\end{center}

\end{intthm}

\begin{figure}[h]
\begin{picture}(350,50)(0,-10)
\multiput(40,0)(20,0){1}{\circle*{3}}

\multiput(15,0)(4,0){3}{\circle*{1}}
\put(25,0){\textcolor{black}{\line(1,0){30}}}
\multiput(57,0)(4,0){3}{\circle*{1}}
\linethickness{.1pt}
\put(40,0){\line(0,1){40}}
\put(40,20){\circle*{3}}
\put(40,40){\circle*{3}}

\put(32,-15){(A)}
\multiput(190,0)(20,0){1}{\circle*{3}}
\put(175,0){\textcolor{black}{\line(1,0){30}}}
\multiput(165,0)(4,0){3}{\circle*{1}}
\multiput(210,0)(4,0){3}{\circle*{1}}
\put(190,0){\textcolor{black}{\line(1,3){7}}}
\put(190,0){\textcolor{black}{\line(-1,3){7}}}
\put(183,21){\circle*{3}}
\put(197,21){\circle*{3}}
\put(182,-15){(B)}

\multiput(340,0)(20,0){1}{\circle*{3}}
\put(325,0){\textcolor{black}{\line(1,0){30}}}

\multiput(315,0)(4,0){3}{\circle*{1}}
\multiput(360,0)(4,0){3}{\circle*{1}}

\put(340,0){\line(0,1){20}}
\put(340,20){\circle*{3}}
\put(332,-15){(C)}
\end{picture}
\caption{3 types of legal limbs} \label{fig:3limbs1}
\end{figure}
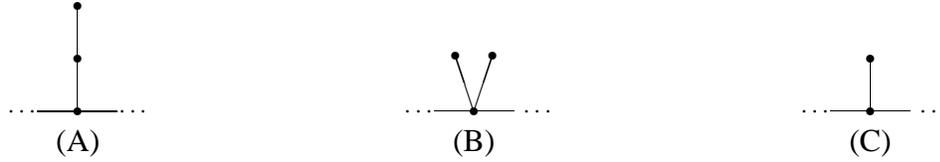

The rest of this paper is structured as follows:
In Section~\ref{background} we provide necessary definitions and background to state our main results
precisely.  Section~\ref{tori} contains results with 
regard to $C_m\Box C_n$ and Section~\ref{diametrical} contains results concerning diametrical 
graphs including a classification of diametrical trees.  
We conclude our paper with Section~\ref{open} which states a few open problems.

\section{Background on Broadcasts in Graphs}\label{background}

Let 
$G=(V,E)$ be a graph with vertex set $V$ and edge set $E$.  
For any $v \in V$ we call the set $N(v)=\{u\in V|\;uv\in E\}$ the \emph{open neighborhood of $v$}. Likewise, the \emph{closed 
neighborhood of $v$} is the set $N[v]=N(v)\cup \{v\}$. We say that a vertex $u$ is a \emph{neighbor} of $v$ if $u
\in N[v]$.
 A \emph{dominating set} is a collection $S \subseteq V$ of vertices in $V$ 
 such that every vertex $v \in V$ is either in $S$ or 
it has a 
neighbor in $S$. 
The cardinality of the smallest possible dominating set, denoted 
$\gamma(G)$, is called the \emph{domination number} of $G$.  
A set $S$ is called a \emph{$\gamma$-set} if $S$ is a dominating set of $G$
with $|S| = \gamma(G)$. We say that a dominating set $S$ 
is \emph{minimal} if removing any vertex from $S$ results in a set
that no longer dominates $G$. The cardinality of the largest possible minimal dominating set is 
called the \emph{upper 
domination number} of $G$, denoted 
$\Gamma(G)$. A set $S$ is called a \emph{$\Gamma$-set} if $S$ is a 
minimal dominating set of $G$
with $|S| = \Gamma(G)$.

\begin{example} Let $G$ be the graph depicted in Figure \ref{fig:graph1}.  Then the open and closed 
neighborhoods of $v_3$ are $N(v_3)=\{v_2,v_4,v_5\}$ and $N[v_3]=\{v_2,v_3,v_4,v_5\}$ respectively. The dominating set 
$S_1=\{v_1,v_3\}$ is a $\gamma$-set and the set $S_2 =\{v_2,v_4,v_5\}$ is a $\Gamma$-set for $G$. Thus we see $\gamma(G)=2$ and 
$\Gamma(G)=3.$

\begin{figure}[h]
\begin{picture}(180,50)
\multiput(10,5)(40,0){4}{\circle*{3}}
\put(10,5){\line(1,0){120}}

\put(5,-5){$v_1$}
\put(45,-5){$v_2$}
\put(85,-5){$v_3$}
\put(90,5){\line(0,1){40}}
\put(90,45){\circle*{3}}
\put(83,52){$v_5$}
\put(125,-5){$v_4$}
\put(10,30){\tcr{$S_1$}}
\put(10,5){\tcr{\circle{5}}}
\put(90,5){\tcr{\circle{5}}}
\put(120,30){\tcb{$S_2$}}
\put(50,5){\tcb{\circle{7}}}
\put(130,5){\tcb{\circle{7}}}
\put(90,45){\tcb{\circle{7}}}
\put(50,5){\tcb{\circle{5}}}
\put(130,5){\tcb{\circle{5}}}
\put(90,45){\tcb{\circle{5}}}
\end{picture}
\caption{Two minimal dominating sets in a graph $G$}\label{fig:graph1}
\end{figure}
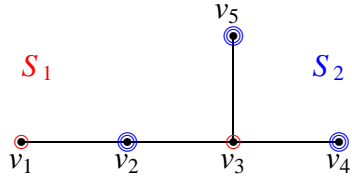
\end{example}

The \emph{distance} between two vertices $v$ and $w$ is the minimum number of edges between $v$ and $w$, denoted $d(v,w)$. 
The \emph{eccentricity} of the vertex $v$ in $G$ is the maximum distance from $v$ 
to any other vertex $u$ in $V$.  We denote this set as 
\[e(v)=\max \{d(v,w):w \in V \}. \]
The \emph{radius of $G$} is the minimum eccentricity among the vertices of $G$ and 
the \emph{diameter of $G$} is the maximum eccentricity among the vertices of $G$.  
We denote them respectively as
\[\rad{G}=\min\{e(v):v \in V\} \hspace{7mm}\text{and}\hspace{7mm} \diam{G}=\max\{e(v):v \in V\}. \]

A \emph{broadcast} on a graph $G$ is a function $f \colon V \rightarrow  \{0,...,\diam{G}\}$ 
such that for every vertex $v \in V(G)$, $f(v) \leq e(v)$.
We let $V_f^+$ denote the set of 
\emph{broadcasting vertices} for $f$. If the broadcast is well understood, 
we simplify the 
notation to $V^+$.
The set of vertices that a
vertex $v \in V$ can \emph{hear} is defined as $H (v) = \left \{u \in V_f^+ ~|~ d(u, v) \leq  f(u) \right \}$, and 
the \emph{broadcast neighborhood} of a broadcasting vertex $v \in V^+_f$ is defined as \[ 
N_f[v] = \left \{ u \in V \colon d(u,v) \leq f(v) \right \} .\]  
We say that a vertex $u$ is a \emph{private $f$-neighbor}, or simply \emph{private neighbor}, of a vertex $v$ if it is in the set 
$\{u\in V\; |\;H(u)=\{v\}\}$.

The \emph{cost} of a broadcast $f$ is the value \[ \cost{f} = \sum_{v \in V_f^+} f(v). \]
 We say that a broadcast \emph{dominates} a graph if every vertex in the graph hears at least one broadcasting vertex.
 That is, for every $v \in V$ there is a $u \in V_f^+$ such that $d(u,v) \leq f(u)$.
 A dominating broadcast $f$ is called \emph{minimal} if decreasing the strength of any broadcasting vertex $v \in V_f^+$ results 
in a non-dominating broadcast. Given a graph $G$, its \emph{broadcast 
domination number}, denoted $\gamma_b(G)$, is defined to be the smallest cost of all minimal dominating broadcasts, that is
\[ \gamma_b(G) = \min \{ \cost{f}\colon f \text{ is a minimal dominating broadcast on } G \},\] and
its \emph{upper broadcast domination number}, denoted $\Gamma_b(G)$, is defined to be the maximum cost of all minimal broadcasts, 
that is, 
\[ \Gamma_b(G)= \max \{ \cost{f}\colon f \text{ is a minimal dominating broadcast on } G \} .\]
A broadcast $f$ is said to be \emph{efficient} if every vertex $v \in V$ hears only one broadcasting vertex.

\begin{example}
Figure \ref{fig:graph2} depicts three minimal dominating sets $f,g,$ and $h$.
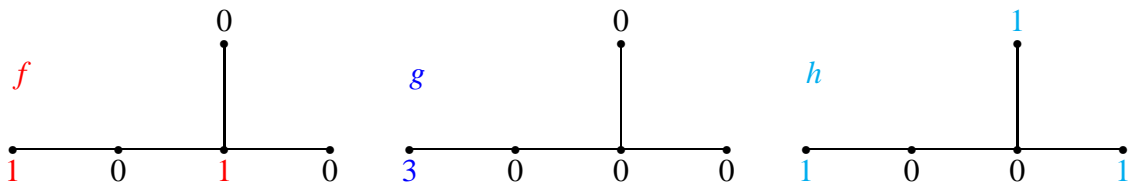
\begin{figure}[h]
\begin{picture}(400,50)
\multiput(10,5)(40,0){4}{\circle*{3}}
\put(10,5){\line(1,0){120}}
\put(7,-7){$\tcr{1}$}
\put(10,30){$\tcr{f}$}
\put(47,-7){$0$}
\put(87,-7){$\tcr{1}$}
\put(90,5){\line(0,1){40}}
\put(90,45){\circle*{3}}
\put(87,50){$0$}
\put(127,-7){$0$} 

\multiput(160,5)(40,0){4}{\circle*{3}}
\put(160,5){\line(1,0){120}}
\put(160,30){$\tcb{g}$}
\put(157,-7){$\tcb{3}$}
\put(197,-7){$0$}
\put(237,-7){$0$}
\put(240,5){\line(0,1){40}}
\put(240,45){\circle*{3}}
\put(237,50){$0$}
\put(277,-7){$0$} 

\multiput(310,5)(40,0){4}{\circle*{3}}
\put(310,5){\line(1,0){120}}
\put(310,30){$\tcc{h}$}
\put(307,-7){$\tcc{1}$}
\put(347,-7){$0$}
\put(387,-7){$0$}
\put(390,5){\line(0,1){40}}
\put(390,45){\circle*{3}}
\put(387,50){$\tcc{1}$}
\put(427,-7){$\tcc{1}$} 

\end{picture}
\caption{Three Minimal Dominating Broadcast of $G$}\label{fig:graph2}
\end{figure}
Labeling the vertices of $G$ as in Figure \ref{fig:graph1}, the three broadcasts are defined as
\begin{itemize}
\item $f(v_1) = f(v_3)=1$ and $f(v_2)=f(v_4)=f(v_5)=0$
\item $g(v_1) =3$ and $g(v_2)=g(v_3)=g(v_4)=g(v_5)=0$
\item $h(v_1) =h(v_4)=h(v_5)=1$ and $h(v_2)=h(v_3)=0$.
\end{itemize} 
We see that $f$ is a $\gamma_b$-broadcast, while $g$ and $h$ are both $\Gamma_b$ broadcasts. Of these three broadcasts, only $g$ 
is efficient, as $v_2 \in N_f[v_1] \cap N_f[v_3]$ and $v_3 \in N_h[v_4] \cap N_h[v_5]$.
\end{example}

The \emph{Cartesian product} of the graphs $G_1$ and $G_2$ is denoted by $G_1\Box G_2$
with vertex set \[ V(G_1\Box 
G_2)=V(G_1)\times V(G_2)=\{(x_1,x_2)~|\; x_i\in V(G_i) \text{ for } i=1,2\}. \]  Any two vertices $(u_1,u_2)$ and $(v_1,v_2)$ are 
adjacent in $G_1\Box G_2$ if and only if 
either 
\begin{itemize}
\item $u_1=v_1$ and $u_2v_2\in E(G_2)$; or 
\item $u_2=v_2$ and $u_1v_1\in E(G_1)$.
\end{itemize}

\section{Broadcasts in Toroidal Grids $C_m \Box C_n$}\label{tori}

Henceforth in this paper, we set $1 \leq m \leq n$ and label the vertices of any $\grid{m}{n}$ or $\TG{m}{n}$ by the convention:
\[ v_{1,1}, v_{1,2}, \ldots, v_{1,n}, 
v_{2,1}, v_{2,2}, \ldots, v_{m,1},v_{m,2}, \ldots 
,v_{m,n}\] 
where the first subscript denotes the row the vertex is in, the second subscript 
denotes the column the vertex is in, and $v_{1,1}$ is in the upper left corner while $v_{m,n}$ is in the lower right.

In this section we find the exact values for $\Gamma(\TG{m}{n})$ and $\Gamma_b(\TG{m}{n})$ of $m \times n$ toroidal grids and 
compare these results with those already existing in the literature on (broadcast) domination theory of toroidal grids. 
We start by stating an obvious relationship between the domination theory of grids and toroidal grids.

\begin{lemma} \label{lemma:gd}
Let $\Delta$ be any domination number.  
Then \[ \Delta (\TG{m}{n}) \leq \Delta (\grid{m}{n}). \]
\end{lemma}
\begin{proof}
Let $S$ be a $\Delta$-dominating set for $\grid{m}{n}$.   Adding edges to connect the vertices $v_{1,i}$ to $v_{m,i}$ for all $i\in\{1,\dots, n\}$  and 
$v_{i,1}$ 
to $v_{i,n}$  for all $i\in\{1,\dots, m\}$ in $\grid{m}{n}$ yields the toroidal grid $\TG{m}{n}$.  The set $S$ also dominates $\TG{m}{n}$. (Note $S$ may not be 
a 
minimal $\Delta$-dominating set for $\TG{m}{n}$.) Therefore $\Delta(\TG{m}{n}) \leq \Delta( \grid{m}{n})$. 
\end{proof}

The domination number of $\TG{m}{n}$ for $m =3,4,$ and $5$ was first considered by Klavzar and Seifter in 1995 
, where they showed the following results hold for $n\geq 4$ \cite[Theorems 2.3 - 2.5]{KS95}:
\begin{align*}
\gamma(\TG{3}{n}) &= n - \left \lfloor \frac{n}{4} \right \rfloor \\
\gamma(\TG{4}{n}) &= n \\
\gamma(\TG{5}{n}) &= \begin{cases}n  & n=5k\\
n+1 & n\in\{5k+1,5k+2,5k+4\}\end{cases}\\
\gamma(\TG{5}{5k+3})&\leq 5(k+1).
\end{align*}

These results provide equations or lower bounds on the domination number of the toroidal graphs $\TG{m}{n}$ for $m =3,4,$ and 
$5$.  We provide equations for the upper domination number of toroidal graphs in Theorem~\ref{theorem:GammaTG}, but first in the 
following theorem we 
prove a formula for the upper domination number of the toroidal graph $\TG{3}{n}$.

\begin{theorem}
 The upper domination number of $\TG{3}{n}$ is given by 
 \[ \Gamma(C_3 \Box
C_n) = n. \]
\end{theorem}
\begin{proof} Define a broadcast $f$ on $\TG{3}{n}$ by $f(v_{2,i})=1$ and $f(v)=0$ for all other $v \in \TG{m}{n}$, as shown in 
Figure~\ref{Fig:C3Cn}. 
Since each broadcasting vertex $v_{j,i} \in V_f^+$ has a private neighbor in its column, this broadcast is minimal \cite[Theorem 
3]{DEHHH05}. This implies that $\Gamma(\TG{3}{n}) 
\geq n$.  
To see that $\Gamma(\TG{3}{n}) \leq n$, note that if we place more than one dominating vertex in any column, then the resulting 
dominating set is not minimal.  Thus any minimal dominating set can have at most one vertex in each column. This proves that 
$\Gamma(\TG{3}{n})=n$.
\end{proof}

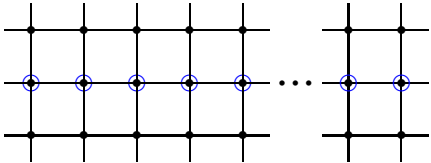
\begin{figure}[h] 
\begin{picture}(400,50)(-30,0)
\multiput(100,10)(20,0){5}{\circle*{3}}
\multiput(100,30)(20,0){5}{\circle*{3}}
\multiput(100,50)(20,0){5}{\circle*{3}}
\multiput(100,30)(20,0){5}{\textcolor{blue}{\circle{6}}}
\put(90,10){\line(1,0){100}}
\put(90,30){\line(1,0){100}}
\put(90,50){\line(1,0){100}}
\put(100,0){\line(0,1){60}}
\put(120,0){\line(0,1){60}}
\put(140,0){\line(0,1){60}}
\put(160,0){\line(0,1){60}}
\put(180,0){\line(0,1){60}}

\put(210,10){\line(1,0){40}}
\put(210,30){\line(1,0){40}}
\put(210,50){\line(1,0){40}}
\put(220,0){\line(0,1){60}}
\put(240,0){\line(0,1){60}}
\multiput(220,10)(20,0){2}{\circle*{3}}
\multiput(220,30)(20,0){2}{\circle*{3}}
\multiput(220,50)(20,0){2}{\circle*{3}}
\multiput(220,30)(20,0){2}{\textcolor{blue}{\circle{6}}}
\multiput(195,30)(05,0){3}{\circle*{2}}
\end{picture}
\caption{A dominating set of $\TG{3}{n}$} \label{Fig:C3Cn}
\end{figure}

The following lemma will prove useful in subsequent proofs. 

\begin{lemma} \label{lemma:2by2}
A minimal dominating set of a $2$ by $2$ grid $\grid{2}{2}$ can contain at most $2$ vertices. 
\end{lemma}
\begin{proof}
Choose any vertex $v$ in $\grid{2}{2}$ to start building a dominating set. Let $f(v)=1$. Then $v$ and its two neighbors are 
dominated by 
$v$.  Choose any vertex $u\ne v$ in $\grid{2}{2}$ and let $f(u)=1$. This action dominates the remaining vertex.  Therefore, the 
largest minimal dominating set 
for $\grid{2}{2}$ contains at most $2$ vertices.
\end{proof}

\begin{theorem} \label{theorem:GammaTG}
The upper domination number of $\TG{m}{n}$ is \[ \Gamma( \TG{m}{n}) = \begin{cases}  \frac{m 
\cdot 
n}{2} & \text{ if } m,n \text{ even } \\ 
\frac{m \cdot(n - 1)}{2} & \text{ if } m \text{ even}, n \text{ odd } \\
\frac{(m-1) \cdot n }{2} & \text{ if } m \text{ odd},  n \text{ even } \\
\frac{(m-1) \cdot (n-1) }{2} + 1 & \text{ if } m,n \text{ odd. } \\

\end{cases} \]  

\end{theorem}
\begin{proof}
Consider an $m \times n$ grid to be the integer lattice. Recall that we label the vertices of the graph in a grid-like fashion 
$$\{v_{1,1},v_{1,2},\dots,v_{1,n},v_{2,1},v_{2,2},\dots,v_{m,1},\dots,v_{m,n}\}.$$  
We proceed by cases to construct a minimal dominating set of maximal cardinality.

\underline{Case 1:} Assume that $m,n$ are even.  Let 
$$V^+=\{v_{1,1},v_{1,3},\dots,v_{1,n-1},v_{2,2},\dots,v_{2,n},\dots,v_{m,2},\dots,v_{m,n}\}.$$  This set is a dominating set 
because each node in this set dominates its four neighboring nodes. We see this illustrated in Figure \ref{figure:cases}(a). 

\underline{Case 2:}  Assume $m$ is even and $n$ is odd.   Let 
$$V^+=\{v_{1,1},v_{1,3},\dots,v_{1,n-2},v_{2,2},\dots,v_{2,n-1},\dots,v_{m,2},\dots,v_{m,n-1}\}.$$ By Case 1, $V^+$ dominates the 
vertices in the subgraph $G_{m,n-1}$.  The vertices in the last column, $\{v_{1,n},v_{2,n},\dots,v_{m,n}\}$, are dominated in 
$\TG{m}{n}$ by the vertices $\{v_{1,1},v_{2,n-1},v_{1,3},\dots,v_{m,n-1}\}$ respectively.  We see this illustrated in Figure~\ref{figure:cases}(b).

\underline{Case 3:} Assume $m$ is odd and $n$ is even.  Let 
$$V^+=\{v_{1,1},v_{1,3},\dots,v_{1,n-1},v_{2,2},\dots,v_{2,n},\dots,v_{m-1,2},\dots,v_{m-1,n}\}.$$ By Case 1, $V^+$ dominates the 
vertices in the subgraph $G_{m-1,n}$.  The vertices in the last row, $\{v_{m,1},v_{m,2},\dots,v_{m,n}\}$, are dominated  by the 
vertices $\{v_{1,1},v_{m-1,2},v_{1,3},\dots,v_{m-1,n}\}$ in $\TG{m}{n}$.

\underline{Case 4:} Assume $m$ and $n$ are both odd.  Let 
$$V^+=\{v_{1,1},v_{1,3},\dots,v_{1,n-2},v_{2,2},\dots,v_{2,n-1},\dots,v_{m-1,2},\dots,v_{m-1,n-1},v_{m,n}\}.$$  By Cases 2 and 3, 
$V^+$ dominates the vertices in the subgraphs $G_{m,n-1}$ and $G_{m-1,n}$.  The vertex $v_{m,n}$ is in the dominating set itself 
and so is covered.  We see this in Figure~\ref{figure:cases}(c). 

Applying Lemma~\ref{lemma:2by2} to each case shows that there are no larger minimal dominating sets for each case.
\end{proof}

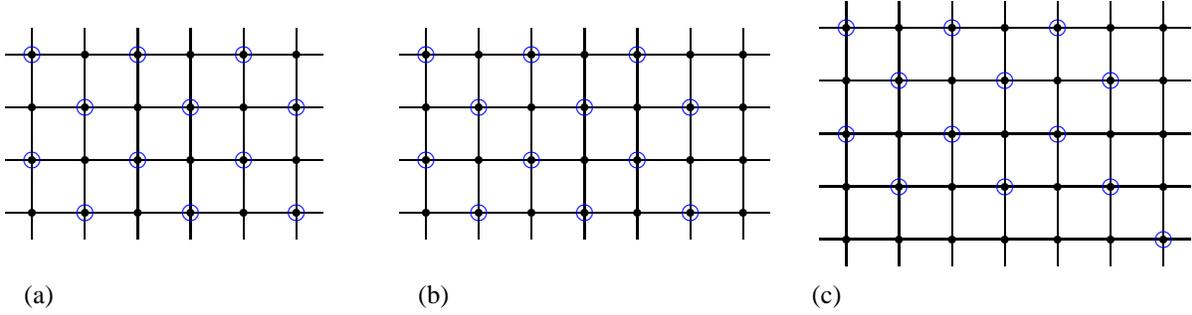
\begin{figure}
\begin{subfigure}[]{}
\begin{picture}(140,100)(00,-10)
\multiput(00,10)(20,0){6}{\circle*{3}}
\multiput(00,30)(20,0){6}{\circle*{3}}
\multiput(00,50)(20,0){6}{\circle*{3}}
\multiput(00,70)(20,0){6}{\circle*{3}}

\put(-10,10){\line(1,0){120}}
\put(-10,30){\line(1,0){120}}
\put(-10,50){\line(1,0){120}}
\put(-10,70){\line(1,0){120}}
\put(00,0){\line(0,1){80}}
\put(20,0){\line(0,1){80}}
\put(40,0){\line(0,1){80}}
\put(60,0){\line(0,1){80}}
\put(80,0){\line(0,1){80}}
\put(100,0){\line(0,1){80}}
\multiput(00,70)(20,20){1}{\textcolor{blue}{\circle{6}}}
\multiput(00,30)(20,20){3}{\textcolor{blue}{\circle{6}}}
\multiput(20,10)(20,20){4}{\textcolor{blue}{\circle{6}}}
\multiput(60,10)(20,20){3}{\textcolor{blue}{\circle{6}}}
\multiput(100,10)(20,20){1}{\textcolor{blue}{\circle{6}}}
\end{picture}
\end{subfigure}
\begin{subfigure}[]{}
\begin{picture}(140,100)(30,-10)
\multiput(30,10)(20,0){7}{\circle*{3}}
\multiput(30,30)(20,0){7}{\circle*{3}}
\multiput(30,50)(20,0){7}{\circle*{3}}
\multiput(30,70)(20,0){7}{\circle*{3}}
\put(20,10){\line(1,0){140}}
\put(20,30){\line(1,0){140}}
\put(20,50){\line(1,0){140}}
\put(20,70){\line(1,0){140}}
\put(30,0){\line(0,1){80}}
\put(50,0){\line(0,1){80}}
\put(70,0){\line(0,1){80}}
\put(90,0){\line(0,1){80}}
\put(110,0){\line(0,1){80}}
\put(130,0){\line(0,1){80}}
\put(150,0){\line(0,1){80}}
\multiput(30,70)(20,20){1}{\textcolor{blue}{\circle{6}}}
\multiput(30,30)(20,20){3}{\textcolor{blue}{\circle{6}}}
\multiput(50,10)(20,20){4}{\textcolor{blue}{\circle{6}}}
\multiput(90,10)(20,20){3}{\textcolor{blue}{\circle{6}}}
\multiput(130,10)(20,20){1}{\textcolor{blue}{\circle{6}}}
\end{picture}
\end{subfigure}
\begin{subfigure}[]{}
\begin{picture}(150,100)(80,00)
\multiput(90,10)(20,0){7}{\circle*{3}}
\multiput(90,30)(20,0){7}{\circle*{3}}
\multiput(90,50)(20,0){7}{\circle*{3}}
\multiput(90,70)(20,0){7}{\circle*{3}}
\multiput(90,90)(20,0){7}{\circle*{3}}
\put(80,10){\line(1,0){140}}
\put(80,30){\line(1,0){140}}
\put(80,50){\line(1,0){140}}
\put(80,70){\line(1,0){140}}
\put(80,90){\line(1,0){140}}
\put(90,00){\line(0,1){100}}
\put(110,00){\line(0,1){100}}
\put(130,00){\line(0,1){100}}
\put(150,00){\line(0,1){100}}
\put(170,00){\line(0,1){100}}
\put(190,00){\line(0,1){100}}
\put(210,00){\line(0,1){100}}
\multiput(90,90)(20,20){1}{\textcolor{blue}{\circle{6}}}
\multiput(90,50)(20,20){3}{\textcolor{blue}{\circle{6}}}
\multiput(110,30)(20,20){4}{\textcolor{blue}{\circle{6}}}
\multiput(150,30)(20,20){3}{\textcolor{blue}{\circle{6}}}
\multiput(190,30)(20,20){1}{\textcolor{blue}{\circle{6}}}
\multiput(210,10)(20,20){1}{\textcolor{blue}{\circle{6}}}
\end{picture}
\end{subfigure}
\caption{An example of Cases 1, 2, and 4 in Theorem~\ref{theorem:GammaTG}.} \label{figure:cases}
\end{figure}

We now consider the upper broadcast domination number $\Gamma_b$ of cycles and products of cycles.

\begin{theorem} \label{thm:CycleGammab}
Let $C_n$ denote the cycle graph with $n$ nodes. 
If $n=3$, then $\Gamma_b(C_3)=1$.
If $n >3$, then \[ \Gamma_b(C_n) = \begin{cases} n-2 & \text{ if } n \text{ is } even \\ n-3 & \text{ if } n \text{ is } odd. 
\end{cases}   \]
\end{theorem}

\begin{proof}
The result for $n=3$ becomes clear when $V_f^+=\{v_i\}$ with $f(v_i)=1$. 
Now, assume $n>3$. Let $g: V \rightarrow \{0,1, \ldots, \diam{C_n}\}$ denote an arbitrary minimal broadcast on $C_n$,
 and let $V_g^+$ denote the set of broadcasting vertices.
 From~\cite[Theorem 3]{DEHHH05}, if $v \in V_g^+$, then $v$ has a private $g$-neighbor (denoted $v_p$) such that either 
\begin{enumerate}[\rm (i)]
\item $g(v)=d(v, v_p)$, or 
\item $g(v)=1$ 
and $v =v_p$. 
\end{enumerate}

As in the proof of Theorem 5 by Dunbar et al. ~\cite[Theorem 5]{DEHHH05}, we define
a function  $\varepsilon:V_g^+ \rightarrow E$ from the broadcasting vertices $V_g^+$ into the edge set $E$ of $C_n$ as follows: 
\begin{itemize}
\item if $v  \in  V_g^+$ satisfies (i), then $\varepsilon(v)$ 
is the set of all edges that lie on the geodesic path between $v$ and $v_p$. Hence $| \varepsilon(v)| \geq g (v)$;
\item if $v$ satisfies (ii), then $\varepsilon(v) = \{e_v \}$, where $e_v$
is any edge incident with $v$.
\end{itemize}

As stated in the proof of~\cite[Theorem 5]{DEHHH05} we know 
$\cost{g} \leq \sum_{v \in V_g^+} |\varepsilon(v)|$, and for 
any pair of 
distinct 
vertices $u,v \in V_g^+$ the paths
$\varepsilon(u) \cap \varepsilon(v) = \emptyset$ are disjoint. 
We use these two facts to prove that \begin{equation} \cost{g} \leq \sum_{v \in V_g^+} |\varepsilon(v)| \leq 
\begin{cases} 
n-2 
& \text{ when } n \text{ is even} \\ n-3 & \text{ when } n \text{ is odd. } \end{cases}  \label{eqn:main1} \end{equation}

Note that if $g$ is a minimal broadcast on $C_n$ with $V_g^+ = \{v\}$,
then $\cost{g} = g(v) \leq \diam{C_n} = \lfloor \frac{n}{2} \rfloor$.  So any minimal broadcast 
$g: C_n \rightarrow \{0,1, 
\ldots, \diam{C_n}\} $ with $\cost{g} \geq \lfloor \frac{n}{2} \rfloor $ must contain
two or more broadcasting vertices.  
Also note that if $g(v)= \diam{C_n} = \lfloor \frac{n}{2} \rfloor$ for some $v\in C_n$,
then every vertex in $C_n$ hears the broadcast from $v$.  
Hence any minimal broadcast $g$ with two or more broadcasting
vertices must have $g(v) < \diam{C_n} $ for 
each $v \in V_g^+$. 

To prove $\cost{g} \leq n-2$, we
assume for the sake of contradiction that there is a minimal broadcast $g:C_n 
\rightarrow \{0,1, \ldots, \diam{C_n} \}$ such 
that $\cost{g}  > n-2$.  Then $\sum_{v \in V_g^+} |\varepsilon(v)| > n-2.$  
But any collection of more than $n-2$ edges in $C_n$ contains a single path of 
length $n-1$, and if the image of $\varepsilon$ is a path $ P_{n-1}$ then $V_{g}^+$ 
contains only a single vertex.  Thus $\cost{g} \leq \lfloor \frac{n}{2}\rfloor <n-1$.
Having arrived at a contradiction, we conclude that  $\cost{g} \leq n-2$ for
any minimal broadcast $g$ and
$\Gamma_b(C_n) \leq n-2$.

Next we assume that $n$ is odd and show that $\cost{g} \leq n-3$ for any minimal
broadcast on $C_n$.  Assume for sake of contradiction, that $g$ is a minimal broadcast on $C_n$ 
with $\cost{g} =n-2$.  Then $\cup_{v \in V_g^+} \varepsilon(v)$ is a disjoint
union of two or more paths in $C_n$.
Up to symmetry, there is only one way to place two or more disjoint paths of length at most
$ \lfloor \frac{n}{2} \rfloor$ on $C_n$
whose union contains $n-2$
edges: that is to place two paths $\varepsilon(v) $ and $\varepsilon(u)$
with $|\varepsilon(v)| =\frac{n-1}{2}$ and $|\varepsilon(u)| = \frac{n-3}{2}$
such that these paths have an edge separating them on each side.
However, in this situation, $g(v)= \diam{C_n} =\frac{n-1}{2},$ and the broadcast
is not minimal.  
Hence $\Gamma_b(C_n)\leq n-3$ when $n$ is 
odd.

We have proved that any minimal broadcast $g$ on $C_n$ must 
satisfy the inequalities in Equation \eqref{eqn:main1}, we now
define minimal broadcasts $f_e$ and $f_o$ on $C_n$ such that 
$\cost{f_e} =n-2$ when $n$ is even and $\cost{f_o} = n-3$ when $n$ is 
odd.  
When $n$ is even we define $f_e(v_1) = f_e(v_n) = \frac{n-2}{2}$. 
This broadcast is minimal because $v_1$ has a private $f_e$-neighbor 
$v_{n/2}$ and $v_n$ has a private $f_e$-neighbor 
$v_{(n+2)/2}$.  
When $n$ is odd we define $f_o(v_2) = f_o(v_n) = \frac{n-3}{2}$.  This broadcast is minimal because $v_2$ has a 
private $f_o$-neighbor $v_{(n+1)/2}$ and $v_n$ has a private $f_o$-neighbor 
$v_{(n+3)/2}$.  
\end{proof}

\begin{example}
We demonstrate Theorem~\ref{thm:CycleGammab} below in the cases where $n=8$ and $n=7$.  Note that for the first graph with $n=8$, 
we have $V_{f_e}^+ =\{v_1, v_8\}$ with $f_e(v_1)=f_e(v_8)=3$.  
In this case the private neighbor of $v_1$ is the vertex $v_4$ 
and the private neighbor of $v_8$ is the vertex $v_5$.

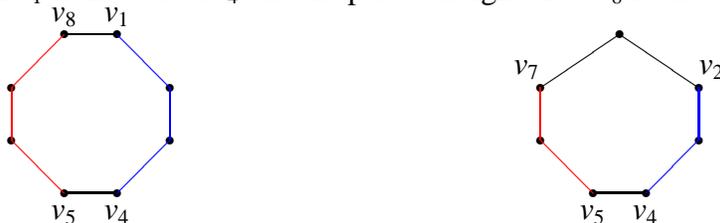
\begin{picture}(400,80)
\multiput(100,10)(20,0){2}{\circle*{3}}
\multiput(80,30)(60,0){2}{\circle*{3}}
\multiput(80,50)(60,0){2}{\circle*{3}}
\multiput(100,70)(20,0){2}{\circle*{3}}
\put(100,10){\line(1,0){20}}
\put(100,70){\line(1,0){20}}
\put(140,50){\textcolor{blue}{\line(-1,1){20}}}
\put(140,30){\textcolor{blue}{\line(-1,-1){20}}}
\put(140,30){\textcolor{blue}{\line(0,1){20}}}
\put(80,50){\textcolor{red}{\line(1,1){20}}}
\put(100,10){\textcolor{red}{\line(-1,1){20}}}
\put(80,30){\textcolor{red}{\line(0,1){20}}}

\put(95,75){$v_8$}
\put(95,0){$v_{5}$}
\put(115,75){$v_1$}
\put(115,0){$v_{4}$}

\multiput(300,10)(20,0){2}{\circle*{3}}
\multiput(280,30)(60,0){2}{\circle*{3}}
\multiput(280,50)(60,0){2}{\circle*{3}}
\multiput(310,70)(20,0){1}{\circle*{3}}
\put(300,10){\line(1,0){20}}

\put(340,50){\textcolor{black}{\line(-3,2){30}}}
\put(340,30){\textcolor{blue}{\line(-1,-1){20}}}
\put(340,30){\textcolor{blue}{\line(0,1){20}}}
\put(280,50){\textcolor{black}{\line(3,2){30}}}
\put(300,10){\textcolor{red}{\line(-1,1){20}}}
\put(280,30){\textcolor{red}{\line(0,1){20}}}

\put(270,55){$v_7$}
\put(340,55){$v_2$}
\put(295,0){$v_{5}$}
\put(315,0){$v_{4}$}
\end{picture}

For the second graph $C_7$,
we let $V_{f_0}^+=\{v_2, v_7\}$ with $f_o(v_2)=f_o(v_7) =2$.  
In this case the private neighbor of $v_2$ 
is the vertex $v_4$ and the private vertex of $v_7$ is the vertex $v_5$.

\end{example}

With Theorem \ref{thm:CycleGammab} in hand, we are ready to prove the second
main result of this section.

\begin{theorem} \label{theorem:TGIF}
 For any $3 \leq m \leq n$, the upper broadcast domination number of $\TG{m}{n}$ is
 \[ \Gamma_b ( \TG{m}{n}) = m \cdot \Gamma_b(C_n). \]
\end{theorem}

\begin{proof}
Let $f$ be the following broadcast on $\TG{m}{n}$, where $3\leq m\leq n$
$$f(v)=\begin{cases}
\Gamma_b(C_n)\hspace{3mm}\text{if }v\in\{v_{j,k} \hspace{1mm}|\hspace{1mm} j\in\{1,2,\dots,m\} \text{ and } k\in\{1,2\}\} \\
0\hspace{12mm} \text{otherwise}.
\end{cases}$$
This broadcast shows that $\Gamma_b ( \TG{m}{n}) \geq m \cdot \Gamma_b(C_n)$.

Suppose there is a minimal broadcast $g$ such that $\cost{g}>\cost{f}=m\Gamma_b(C_n)$.  Then 
by the pigeonhole 
principle, 
there must exist at least one row of vertices, 
say $\{v_{i,1},\dots,v_{i,n}\}$ such that  the cost of the broadcast in that particular row is greater than $\Gamma_b(C_n)$.  
Then 
the graph contains a subgraph of a cycle $C_n$ with minimal broadcast more than $\Gamma_b(C_n)$.  This contradicts 
Theorem~\ref{thm:CycleGammab}.  Therefore  $ \Gamma_b ( \TG{m}{n}) = m\cdot \Gamma_b(C_n)  $ for all $1 \leq m \leq n$. 
\end{proof}

\section{Diametrical Graphs}\label{diametrical}

We call a graph $G$ a \emph{diametrical graph} if 
$\diam{G}=\Gamma_b(G)$.  Diametrical graphs have the property that their most costly minimal broadcast can be obtained by a single 
broadcasting node $v$ lying at one end of a diametrical path in $G$ by setting $f(v)=\diam{G}$.    (However, there may be many more minimal broadcasts whose cost is $\diam{G}$.)  
Predictably, we say a graph $G$ is 
\emph{non-diametrical} if it is not diametrical.

The goal of this section is to prove Theorem \ref{thm:MainDiametrical}, which classifies all diametrical trees by showing that 
they form a subfamily of lobster graphs with special restrictions placed on the shape and spacing of their \emph{limbs}, that is, subtrees protruding from the central diametrical path of the lobster graph.  
This answers an open question posed by~\cite[Problem 7]{Mynhardt}.

\begin{theorem}\label{thm:MainDiametrical}
 A tree $T$ is diametrical if and only if it is a lobster graph containing only limbs of types $(A)$, $(B)$, and $(C)$ depicted in 
Figure \ref{fig:3limbs},  the number of limbs is less than half the diameter of the graph, and the distance between each pair of 
adjacent limbs or an endpoint $e_i$ satisfies the following inequalities.

\begin{center}
\begin{tabular}{| c|c| c| } \hline 
$d(A,A) \geq 4$ & $d(A,B) \geq 3$ & 
$d(A,C) \geq 3$ \\ $d(B,B) \geq 3$ &
$d(B,C) \geq 2$ & 
$d(C,C) \geq 2$  \\ 
$d(e_i,A) \geq 2$ & $d(e_i,B) \geq 2$ & $d(e_i,C) \geq 1$  \\ \hline
\end{tabular}
\end{center}
\end{theorem}

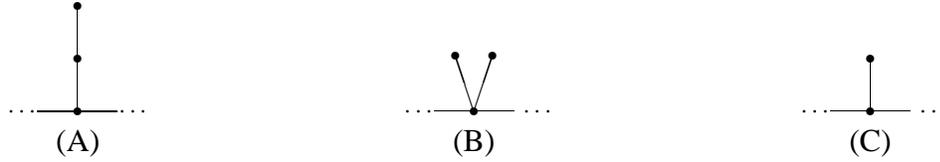
\begin{figure}[h]
\begin{picture}(350,50)(0,-10)
\multiput(40,0)(20,0){1}{\circle*{3}}

\multiput(15,0)(4,0){3}{\circle*{1}}
\put(25,0){\textcolor{black}{\line(1,0){30}}}
\multiput(57,0)(4,0){3}{\circle*{1}}
\linethickness{.1pt}
\put(40,0){\line(0,1){40}}
\put(40,20){\circle*{3}}
\put(40,40){\circle*{3}}

\put(32,-15){(A)}
\multiput(190,0)(20,0){1}{\circle*{3}}
\put(175,0){\textcolor{black}{\line(1,0){30}}}
\multiput(165,0)(4,0){3}{\circle*{1}}
\multiput(210,0)(4,0){3}{\circle*{1}}
\put(190,0){\textcolor{black}{\line(1,3){7}}}
\put(190,0){\textcolor{black}{\line(-1,3){7}}}
\put(183,21){\circle*{3}}
\put(197,21){\circle*{3}}
\put(182,-15){(B)}

\multiput(340,0)(20,0){1}{\circle*{3}}
\put(325,0){\textcolor{black}{\line(1,0){30}}}

\multiput(315,0)(4,0){3}{\circle*{1}}
\multiput(360,0)(4,0){3}{\circle*{1}}

\put(340,0){\line(0,1){20}}
\put(340,20){\circle*{3}}
\put(332,-15){(C)}
\end{picture}
\caption{3 types of legal limbs} \label{fig:3limbs}
\end{figure}

Before proceeding to the proof of Theorem \ref{thm:MainDiametrical} we give an example of how to apply the theorem.

\begin{example} 
Figure 
\ref{Diamexamples} shows a diametrical tree on the left and a non-diametrical tree on the right. The tree on the left contains all 
three types of limbs (A), (B), and (C) that are allowed in a diametrical tree 
and conforms to
  the spacing constraints described in Theorem \ref{thm:MainDiametrical}:

\[ d(e_1, A) = 2,\quad d(A,C_1) =3, \quad d(C_1, B) = 3, \quad d(B,C_2) = 3, \quad d(C_2,e_2) =1.  \]
The tree on the right is not diametrical as it contains an illegal limb $X$ of length $3$; it also contains a pair of legal 
limbs of types $B$ and $C$ that are too close together with $d(B,C) = 1$.

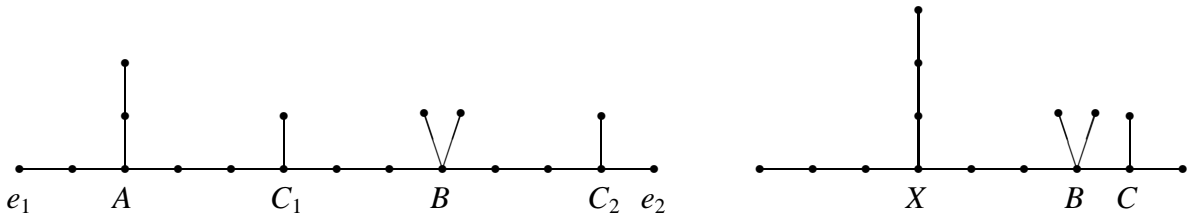
\begin{figure}[h]
\begin{picture}(320,60)(60,-10)
\multiput(00,00)(20,0){13}{\circle*{3}}
\put(00,00){\textcolor{black}{\line(1,0){240}}}
\put(40,00){\line(0,1){40}}
\put(-05,-15){$e_1$}
\put(40,20){\circle*{3}}
\put(40,40){\circle*{3}}
\put(35,-15){$A$}
\put(100,00){\line(0,1){20}}
\put(100,20){\circle*{3}}
\put(95,-15){$C_1$}

\put(160,00){\line(-1,3){7}}
\put(160,00){\line(1,3){7}}
\put(153,21){\circle*{3}}
\put(167,21){\circle*{3}}
\put(155,-15){$B$}

\put(220,00){\line(0,1){20}}
\put(220,20){\circle*{3}}
\put(215,-15){$C_2$}
\put(235,-15){$e_2$}

\multiput(280,00)(20,0){9}{\circle*{3}}
\put(280,00){\textcolor{black}{\line(1,0){160}}}

\put(340,00){\line(0,1){60}}
\put(340,20){\circle*{3}}
\put(340,40){\circle*{3}}
\put(340,60){\circle*{3}}
\put(335,-15){$X$}

\put(420,00){\line(0,1){20}}
\put(420,20){\circle*{3}}
\put(415,-15){$C$}

\put(400,00){\line(-1,3){7}}
\put(400,00){\line(1,3){7}}
\put(393,21){\circle*{3}}
\put(407,21){\circle*{3}}
\put(395,-15){$B$}

\end{picture}
\caption{A diametrical tree of diameter 12 and a non-diametrical tree of diameter 8} \label{Diamexamples}
\end{figure}
\end{example}

We note that Theorem \ref{thm:MainDiametrical} generalizes the following result proved by Dunbar et al. in 2006.

\begin{lemma}[\cite{DEHHH05}, Theorem 5]
If $G$ is a graph of size $m$ (containing $m$ edges), then
$\Gamma_b(G) \leq m$ with equality if and only if $G$ is a
nontrivial star or path.
\end{lemma}

\subsection{A Proof of Theorem \ref{thm:MainDiametrical} }\label{sectionfourpointone}
The rest of this section is dedicated to proving Theorem \ref{thm:MainDiametrical} via a series of lemmas. 
We begin by showing that concatenating any two diametrical trees results in another diametrical tree.  
Next we show in Lemma~\ref{lemma:Lobster} that if a tree has a limb with length longer than two, then it cannot be diametrical.  This reduces the number of cases that we need to consider.
We look at the six possible limb variations on a tree when limbs longer than two are not considered.  

Three of the variations result in a tree that is non-diametrical and the remaining three variations result in trees that may be diametrical depending on the spacing between the limbs.   
Lemma~\ref{firstlemma} proves the non-diametrical nature of three of the variations.  Lemmas~\ref{lemma:3typesoflimbs} and~\ref{lemma:et} discuss the three variations that result in trees that may or may not be diametrical.  They show that when a limb of that variation is part of a tree that is otherwise diametrical, it stays diametrical.  
Next we provide a sufficient condition for identifying non-diametrical graphs $G$.  This allows us to prove that a special case 
which on first glance may seem diametrical, is in fact non diametrical.
Lastly, we prove in Lemmas~\ref{lemma:illegalspacing} and~\ref{lemma:limbdistances} the restrictions on the spacing between two limbs of the same or different varieties that result in a diametrical tree.

We begin by 
setting notation that is used for the remainder of the paper. Set $T$ to be a tree with $\diam{T}=d$ and fix a 
diametrical path $D$ in $T$. We say that a node $u \notin D$ \emph{protrudes} from $v$ if $v \in D$ is the closest vertex to $u$ 
of all vertices in $D$.  
Label the vertices in $D$ as $v_0, \ldots, v_d$.  We define a \emph{leaf} in a graph $G$ to be a degree-one vertex in $G$.
For each vertex $v_i$, label the  vertices as in the example below.  

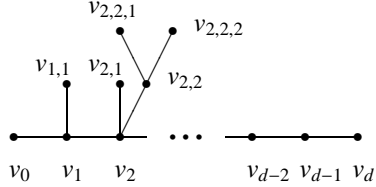
\begin{picture}(320,80)(-20,-20)
\multiput(120,0)(20,0){3}{\circle*{3}}
\put(120,0){\textcolor{black}{\line(1,0){50}}}
\multiput(180,0)(05,0){3}{\circle*{2}}
\put(200,0){\textcolor{black}{\line(1,0){50}}}
\multiput(210,0)(20,0){3}{\circle*{3}}
\put(140,0){\textcolor{black}{\line(0,1){20}}}
\put(160,0){\textcolor{black}{\line(0,1){20}}}
\put(160,0){\line(1,2){20}}
\put(170,20){\line(-1,2){10}}
\multiput(160,20)(2,2){1}{\circle*{3}}
\multiput(140,20)(2,2){1}{\circle*{3}}
\multiput(170,20)(2,2){1}{\circle*{3}}
\multiput(180,40)(2,2){1}{\circle*{3}}
\multiput(160,40)(2,2){1}{\circle*{3}}
\put(118,-15){\footnotesize{$v_0$}}
\put(138,-15){\footnotesize{$v_1$}}
\put(158,-15){\footnotesize{$v_2$}}
\put(208,-15){\footnotesize{$v_{d-2}$}}
\put(228,-15){\footnotesize{$v_{d-1}$}}
\put(248,-15){\footnotesize{$v_d$}}
\put(128,25){\footnotesize{$v_{1,1}$}}
\put(148,25){\footnotesize{$v_{2,1}$}}
\put(178,20){\footnotesize{$v_{2,2}$}}
\put(188,40){\footnotesize{$v_{2,2,2}$}}
\put(148,47){\footnotesize{$v_{2,2,1}$}}
\end{picture}

Notice that when there exist multiple limbs of distance two from the same 
vertex on the path $D$, we add another number to 
its subscript ordering the vertices from left to right.

In our diagrams we use a box with a label $d_i$ in it to denote a diametrical subgraph of the tree containing $d_i$ edges from 
the 
diametrical path $D$.  This means that the most costly minimal broadcast on this subtree has cost $d_i$.  A demonstration of the box notation is on the left and an example using the box notation can be seen on the right.

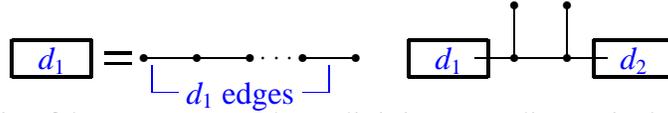
\begin{picture}(350,40)(-90,-15)

 \put(10,-4){\textcolor{blue}{$d_1$}}

\linethickness{1pt}
\put(0,-7){{\line(1,0){30}}}
\put(0,7){{\line(1,0){30}}}
\put(0,-7){{\line(0,1){14}}}
\put(30,-7){{\line(0,1){14}}}

\put(35,2){{\line(1,0){10}}}
\put(35,-2){{\line(1,0){10}}}

\linethickness{.1pt}
\multiput(50,0)(20,0){5}{\circle*{3}}
\put(50,0){\textcolor{black}{\line(1,0){40}}}

\multiput(95,0)(5,0){3}{\circle*{1}}
\put(110,0){\textcolor{black}{\line(1,0){20}}}

\put(53,-3){\tcb{\line(0,-1){10}}}
\put(53,-13){\tcb{\line(1,0){10}}}
\put(66,-17){\tcb{$d_1$ edges}}
\put(110,-13){\tcb{\line(1,0){10}}}
\put(120,-3){\tcb{\line(0,-1){10}}}

\multiput(190,0)(20,0){2}{\circle*{3}}
\put(175,0){\textcolor{black}{\line(1,0){50}}}
 \put(160,-4){\textcolor{blue}{$d_1$}}
 \put(230,-4){\textcolor{blue}{$d_2$}}
 
\linethickness{1pt}
\put(150,-7){{\line(1,0){30}}}
\put(150,7){{\line(1,0){30}}}
\put(150,-7){{\line(0,1){14}}}
\put(180,-7){{\line(0,1){14}}}

\put(220,-7){{\line(1,0){30}}}
\put(220,7){{\line(1,0){30}}}
\put(220,-7){{\line(0,1){14}}}
\put(250,-7){{\line(0,1){14}}}

\linethickness{.1pt}
\put(190,0){\line(0,1){20}}
\put(190,20){\circle*{3}}
\put(210,0){\line(0,1){20}}
\put(210,20){\circle*{3}}
\end{picture}

 The following pair of lemmas prove that adjoining two diametrical trees
with a path of any length will always result a diametrical tree.  
Thus to show 
that a tree is diametrical, it suffices to show that 
the separate pieces of the tree are diametrical.
We say that the path $D=D_1+D_2$ is \emph{concatenated} 
from the paths $D_1$ and $D_2$ if we identify an endpoint from $D_1$ with 
an endpoint of $D_2$.

\begin{lemma} \label{lemma:concatenate}
 Let $T_1$ and $T_2$ be two diametrical trees with diametrical paths $D_1$ and $D_2$ respectively (with diameters $d_1$ and 
$d_2$).
 Then the tree $T=T_1 \cup T_2$ obtained by concatenating the paths $D_1$ and $D_2$ is diametrical.
\end{lemma}

\begin{proof}
 The diameter of $T$ is $\diam{T} = d = d_1+d_2$.  Every broadcast defined on $T_i$ costs at most $d_i$.
 Thus every broadcast on $T$ costs at most $d_1+d_2 = d$.  Hence $\Gamma_b(T) = \diam{T}$.  
\end{proof}

The next lemma shows that connecting diametrical trees with a path results in another diametrical tree.  
\begin{lemma} \label{lemma:pathadd}
 Let $T_1$ and $T_2$ be diametrical trees with diametrical paths $D_1$ and $D_2$ respectively, and let $T$ be a tree obtained by 
concatenating the path $P_n$ with $D_1$ on one end and $D_2$ on the other end
$T$ is diametrical.
\end{lemma}

\begin{proof}
Apply Lemma \ref{lemma:concatenate} twice.
\end{proof}

Our next lemma shows that diametrical trees form a subclass of lobster graphs, which is the first claim in Theorem 
\ref{thm:MainDiametrical}. 

\begin{lemma}\label{lemma:Lobster}
 Let $T$ be a tree with diametral path $D$.  
 Then $T$ is non-diametrical if it contains a vertex $v_i$ on $D$ such that a limb  of length $\ell>2$ protrudes from it.
 \end{lemma}

\begin{proof}
Let $L$ be a limb of length $\ell >2$ protruding from $v_i$.  Let $v$ denote the vertex on $L$ that is distance $\ell$ from
the central path $D$.  Define a broadcast $f$ on $T$ so that 
\begin{align*}
f(v_0)&= i-1\\
f(v_d)&= d-i-1\\
f(v)&=\ell.
\end{align*}
Then $\cost{f}\geq d+\ell-2>2$.  Hence $\Gamma_b(T) > d$ and $T$ is non-diametrical.
\end{proof}

The previous lemma shows that, a priori, there are only a six types of limbs that can protrude from the central path of a 
diametrical tree.  They are the six limbs of length 2 or less shown below.

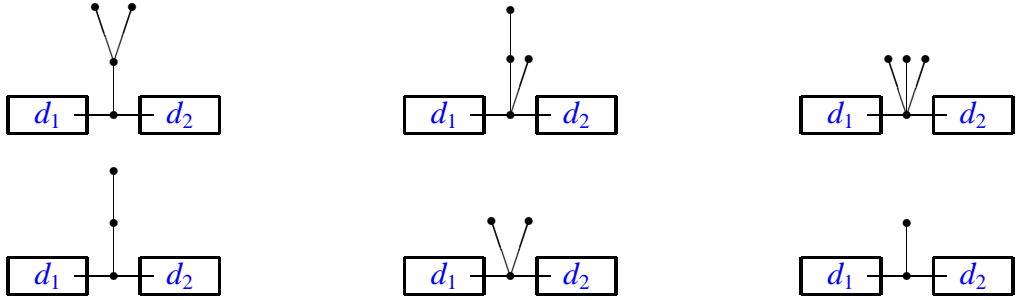
\begin{picture}(350,60)(0,0)

\multiput(40,0)(20,0){1}{\circle*{3}}
\put(25,0){\textcolor{black}{\line(1,0){30}}}
\put(10,-3){\textcolor{blue}{$d_1$}}
\put(60,-3){\textcolor{blue}{$d_2$}}

\linethickness{1pt}
\put(00,-7){{\line(1,0){30}}}
\put(00,7){{\line(1,0){30}}}
\put(00,-7){{\line(0,1){14}}}
\put(30,-7){{\line(0,1){14}}}

\put(50,-7){{\line(1,0){30}}}
\put(50,7){{\line(1,0){30}}}
\put(50,-7){{\line(0,1){14}}}
\put(80,-7){{\line(0,1){14}}}

\linethickness{.1pt}
\put(40,0){\line(0,1){20}}
\put(40,20){\circle*{3}}
\put(33,41){\circle*{3}}
\put(47,41){\circle*{3}}
\put(40,20){\line(1,3){7}}
\put(40,20){\line(-1,3){7}}

\multiput(190,0)(20,0){1}{\circle*{3}}
\put(175,0){\textcolor{black}{\line(1,0){30}}}
\put(160,-3){\textcolor{blue}{$d_1$}}
\put(210,-3){\textcolor{blue}{$d_2$}}

\linethickness{1pt}
\put(150,-7){{\line(1,0){30}}}
\put(150,7){{\line(1,0){30}}}
\put(150,-7){{\line(0,1){14}}}
\put(180,-7){{\line(0,1){14}}}

\put(200,-7){{\line(1,0){30}}}
\put(200,7){{\line(1,0){30}}}
\put(200,-7){{\line(0,1){14}}}
\put(230,-7){{\line(0,1){14}}}

\linethickness{.1pt}
\put(190,0){\line(0,1){40}}
\put(190,21){\circle*{3}}
\put(190,40){\circle*{3}}
\put(197,21){\circle*{3}}
\put(190,0){\line(1,3){7}}

\multiput(340,0)(20,0){1}{\circle*{3}}
\put(325,0){\textcolor{black}{\line(1,0){30}}}
\put(310,-3){\textcolor{blue}{$d_1$}}
\put(360,-3){\textcolor{blue}{$d_2$}}

\linethickness{1pt}
\put(300,-7){{\line(1,0){30}}}
\put(300,7){{\line(1,0){30}}}
\put(300,-7){{\line(0,1){14}}}
\put(330,-7){{\line(0,1){14}}}

\put(350,-7){{\line(1,0){30}}}
\put(350,7){{\line(1,0){30}}}
\put(350,-7){{\line(0,1){14}}}
\put(380,-7){{\line(0,1){14}}}

\linethickness{.1pt}
\put(340,0){\line(0,1){21}}
\put(340,21){\circle*{3}}
\put(340,0){\line(1,3){7}}
\put(340,0){\line(-1,3){7}}
\put(347,21){\circle*{3}}
\put(333,21){\circle*{3}}
\end{picture}

\begin{picture}(350,70)(0,-10)
\multiput(40,0)(20,0){1}{\circle*{3}}
\put(25,0){\textcolor{black}{\line(1,0){30}}}
\put(10,-3){\textcolor{blue}{$d_1$}}
\put(60,-3){\textcolor{blue}{$d_2$}}

\linethickness{1pt}
\put(0,-7){{\line(1,0){30}}}
\put(0,7){{\line(1,0){30}}}
\put(0,-7){{\line(0,1){14}}}
\put(30,-7){{\line(0,1){14}}}

\put(50,-7){{\line(1,0){30}}}
\put(50,7){{\line(1,0){30}}}
\put(50,-7){{\line(0,1){14}}}
\put(80,-7){{\line(0,1){14}}}

\linethickness{.1pt}
\put(40,0){\line(0,1){40}}
\put(40,20){\circle*{3}}
\put(40,40){\circle*{3}}

\multiput(190,0)(20,0){1}{\circle*{3}}
\put(175,0){\textcolor{black}{\line(1,0){30}}}
\put(160,-3){\textcolor{blue}{$d_1$}}
\put(210,-3){\textcolor{blue}{$d_2$}}

\linethickness{1pt}
\put(150,-7){{\line(1,0){30}}}
\put(150,7){{\line(1,0){30}}}
\put(150,-7){{\line(0,1){14}}}
\put(180,-7){{\line(0,1){14}}}

\put(200,-7){{\line(1,0){30}}}
\put(200,7){{\line(1,0){30}}}
\put(200,-7){{\line(0,1){14}}}
\put(230,-7){{\line(0,1){14}}}

\linethickness{.1pt}
\put(190,0){\line(1,3){7}}
\put(190,0){\line(-1,3){7}}
\put(183,21){\circle*{3}}
\put(197,21){\circle*{3}}

\multiput(340,0)(20,0){1}{\circle*{3}}
\put(325,0){\textcolor{black}{\line(1,0){30}}}
\put(310,-3){\textcolor{blue}{$d_1$}}
\put(360,-3){\textcolor{blue}{$d_2$}}

\linethickness{1pt}
\put(300,-7){{\line(1,0){30}}}
\put(300,7){{\line(1,0){30}}}
\put(300,-7){{\line(0,1){14}}}
\put(330,-7){{\line(0,1){14}}}

\put(350,-7){{\line(1,0){30}}}
\put(350,7){{\line(1,0){30}}}
\put(350,-7){{\line(0,1){14}}}
\put(380,-7){{\line(0,1){14}}}

\linethickness{.1pt}
\put(340,0){\line(0,1){20}}
\put(340,20){\circle*{3}}

\end{picture}

The next lemma shows that the top three limbs depicted above (which each contain 3 edges) are not allowed in a diametrical 
tree.

\begin{lemma}\label{firstlemma}
 Let $T$ be a tree.  
 Then $T$ is non-diametrical if it contains any of the following conditions as a subgraph:
 \begin{enumerate}[\rm (i)]

 \item a vertex that is not part of the diametrical path $D$ which has degree greater than two;
 \item  a vertex $v_i$ on $D$ such that protruding from $v_i$ are  the vertices $v_{i,1}$, $v_{i,2}$, and 
$v_{i,j,1}$ for some $j\in\{1,2\}$; or
 \item a vertex $v_i$ on $D$ such that $v_i$ has three or more protrusions from it.
 \end{enumerate}
 \end{lemma}

\begin{proof} 
For each condition, it suffices to find a broadcast $f$ such that $\cost{f}>\diam{T}$. We give one such broadcast for each of the 
three types of graphs.

First we prove (i).  Let $v_i$ be a vertex on $D$ such that the subtree protruding from $v_i$ contains a vertex $v_{i,1}$ of degree greater than 2.  Then there are at least two vertices $v_{i,1,1}$ and $v_{i,1,2}$  protruding from $v_{i,1}$.
Define a broadcast $f$ so that 
\begin{align*}
f(v_0)&=i-1\\
f(v_d)&=d-i\\
f(v_{i,1,1})&=f(v_{i,1,2})=1.
\end{align*}
Then cost$(f)\geq i-1+d-i+1+1=d+1$.
Hence $\Gamma_b(T)>d$ and $T$ is non-diametrical.

To prove (ii), define a broadcast $f$ on $T$ so that
\begin{align*}
f(v_0)&=i+1\\
f(v_d)&=d-i-1\\
f(v_{i,j,1} )&=1.
\end{align*}
Then $\cost{f}\geq d+1>d$.  Hence $\Gamma_b(T)>d$ and $T$ is non-diametrical.

Finally, to prove (iii), 
suppose that $T$ contains a node $v_i$ on $D$ such that $v_i$ has three or more protrusions from it.  Then 
there are vertices 
$v_{i,1},v_{i,2},\dots,v_{i,k}$ with $k\geq 3$.  Define a broadcast $f$ on $T$ so that
 \begin{align*}
 f(v_0)&=i-1\\
 f(v_d)&=d-i-1\\
 f(v_{i,1})& =f(v_{i,2})=\cdots=f(v_{i,k})=1.
 \end{align*}
 Then $\cost{f}\geq d+k-2>d$.  Hence $\Gamma_b(T)>d$ and $T$ is non-diametrical.
 \end{proof}
 
We now show that the limbs of types $(A), (B), $ and $(C)$ are allowed in diametrical trees.   

\begin{lemma}  \label{lemma:3typesoflimbs}
Let $T$ be a tree 
with a branch containing less than 3 edges protruding from vertex $v_i$ of $T$. Suppose that the induced 
subtrees containing $v_0, \ldots, v_{i-1}$ and $v_{i+1}, \ldots, v_d$ are both diametrical.   Then $T$ is diametrical.
\end{lemma}

\begin{proof}
Since the induced subtrees to the left and right of $v_i$ are both diametrical, we may assume without loss of generality that 
those two subgraphs dominated by a broadcast $f(v_0)=i-1=d_1$ and $f(v_d) = d-i-1 = d_2$ (as any other broadcast on these 
subgraphs will be less costly).   
Then there are only finitely many minimal broadcasts on the remaining branch protruding from $v_i$.  We depict the most costly of 
all broadcasts on each branch, and note that the resulting broadcast still satisfies $\cost{f} \leq \diam{T}$.

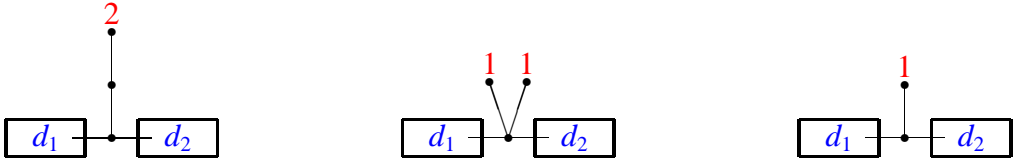
\begin{picture}(350,70)(0,-10)
\multiput(40,0)(20,0){1}{\circle*{3}}
\put(25,0){\textcolor{black}{\line(1,0){30}}}
\put(10,-3){\textcolor{blue}{$d_1$}}
\put(60,-3){\textcolor{blue}{$d_2$}}

\linethickness{1pt}
\put(0,-7){{\line(1,0){30}}}
\put(0,7){{\line(1,0){30}}}
\put(0,-7){{\line(0,1){14}}}
\put(30,-7){{\line(0,1){14}}}

\put(50,-7){{\line(1,0){30}}}
\put(50,7){{\line(1,0){30}}}
\put(50,-7){{\line(0,1){14}}}
\put(80,-7){{\line(0,1){14}}}

\linethickness{.1pt}
\put(40,0){\line(0,1){40}}
\put(40,20){\circle*{3}}
\put(40,40){\circle*{3}}
\put(37,43){\tcr{2}}

\multiput(190,0)(20,0){1}{\circle*{3}}
\put(175,0){\textcolor{black}{\line(1,0){30}}}
\put(160,-3){\textcolor{blue}{$d_1$}}
\put(210,-3){\textcolor{blue}{$d_2$}}

\linethickness{1pt}
\put(150,-7){{\line(1,0){30}}}
\put(150,7){{\line(1,0){30}}}
\put(150,-7){{\line(0,1){14}}}
\put(180,-7){{\line(0,1){14}}}

\put(200,-7){{\line(1,0){30}}}
\put(200,7){{\line(1,0){30}}}
\put(200,-7){{\line(0,1){14}}}
\put(230,-7){{\line(0,1){14}}}

\linethickness{.1pt}
\put(190,0){\line(1,3){7}}
\put(190,0){\line(-1,3){7}}
\put(183,21){\circle*{3}}
\put(197,21){\circle*{3}}
\put(194,24){\tcr{1}}
\put(180,24){\tcr{1}}

\multiput(340,0)(20,0){1}{\circle*{3}}
\put(325,0){\textcolor{black}{\line(1,0){30}}}
\put(310,-3){\textcolor{blue}{$d_1$}}
\put(360,-3){\textcolor{blue}{$d_2$}}

\linethickness{1pt}
\put(300,-7){{\line(1,0){30}}}
\put(300,7){{\line(1,0){30}}}
\put(300,-7){{\line(0,1){14}}}
\put(330,-7){{\line(0,1){14}}}

\put(350,-7){{\line(1,0){30}}}
\put(350,7){{\line(1,0){30}}}
\put(350,-7){{\line(0,1){14}}}
\put(380,-7){{\line(0,1){14}}}

\linethickness{.1pt}
\put(340,0){\line(0,1){20}}
\put(340,20){\circle*{3}}
\put(337,23){\tcr{1}}
\end{picture}

One can easily verify that each of these broadcasts is minimal and that the cost of each broadcast is at most $d$. Hence these 
broadcasts all satisfy $\cost{f}\leq \diam{T}$.  
 \end{proof}

\begin{lemma}\label{lemma:et}
Let $T$ be a diametrical tree. Let $e_1,e_2$ denote the two endpoints of a diametrical path $D$ in $T$.  
The minimum distance between a limb of type $(A), (B),$ or $(C)$ and an endpoint $e_i$ satisfies one of the following equalities:
\[ d(e_i,A) \geq 2 \quad d(e_i,B) \geq 2 \quad d(e_i,C) \geq 1 .\]
\end{lemma}

\begin{proof}
If  $d(e_i,A) < 2$ and $d(e_i,C) < 1$, then the path from the endpoint of the limb $(A)$ or $(C)$
to the other endpoint of $D$ is longer than the diametrical path $D$.  That contradicts our assumptions about $D$.   
Next we show that every possible broadcast defined in a broadcast neighborhood of $e_i$ and $(A)$, or $e_i$ and $(C)$ are 
diametrical when $d(e_i,A) 
\geq 2$ and $d(e_i,C)  \geq 1$.

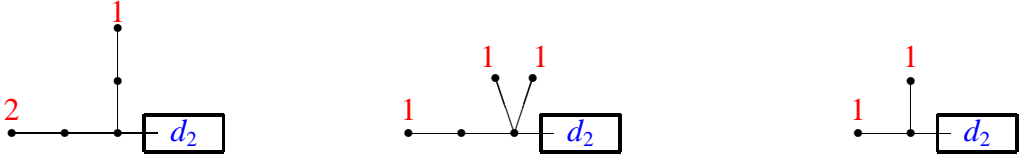
\begin{picture}(350,70)(0,-10)
\multiput(00,0)(20,0){1}{\circle*{3}}
\put(00,0){\textcolor{black}{\line(1,0){55}}}
\put(-3,5){\textcolor{red}{2}}
\put(37,42){\textcolor{red}{1}}
\put(60,-3){\textcolor{blue}{$d_2$}}

\linethickness{1pt}
\put(50,-7){{\line(1,0){30}}}
\put(50,7){{\line(1,0){30}}}
\put(50,-7){{\line(0,1){14}}}
\put(80,-7){{\line(0,1){14}}}

\linethickness{.1pt}
\put(40,0){\line(0,1){40}}
\put(40,20){\circle*{3}}
\put(40,40){\circle*{3}}
\put(20,00){\circle*{3}}
\put(40,00){\circle*{3}}

\multiput(150,0)(20,0){3}{\circle*{3}}
\put(147,5){\textcolor{red}{1}}
\put(197,25){\textcolor{red}{1}}
\put(177,25){\textcolor{red}{1}}
\put(150,0){\textcolor{black}{\line(1,0){55}}}

\put(210,-3){\textcolor{blue}{$d_2$}}

\linethickness{1pt}
\put(200,-7){{\line(1,0){30}}}
\put(200,7){{\line(1,0){30}}}
\put(200,-7){{\line(0,1){14}}}
\put(230,-7){{\line(0,1){14}}}

\linethickness{.1pt}
\put(190,0){\line(1,3){7}}
\put(190,0){\line(-1,3){7}}
\put(183,21){\circle*{3}}
\put(197,21){\circle*{3}}

\multiput(320,0)(20,0){2}{\circle*{3}}
\put(320,0){\textcolor{black}{\line(1,0){35}}}
\put(317,5){\textcolor{red}{1} }
\put(337,25){\textcolor{red}{1} }
\put(360,-3){\textcolor{blue}{$d_2$}}

\linethickness{1pt}

\put(350,-7){{\line(1,0){30}}}
\put(350,7){{\line(1,0){30}}}
\put(350,-7){{\line(0,1){14}}}
\put(380,-7){{\line(0,1){14}}}

\linethickness{.1pt}
\put(340,0){\line(0,1){20}}
\put(340,20){\circle*{3}}
\end{picture}

To see that  $d(e_i,B) \geq 2$, we show that a tree with $d(e_i,B) =1$ is not diametrical, i.e., if either vertex $v_1$ or 
$v_{d-1}$ has degree 4, the tree is not diametrical.  
We show the case where the vertex $v_1$ is of degree at least four and note that the other case follows by symmetry.  Define a 
broadcast $f$ on $T$ so that
\begin{align*}
f(v_0)&=1\\
f(v_{1,1})&=1\\
f(v_{1,2})&=1\\
f(v_d)&=d-2.
\end{align*}
Then $\cost{f}\geq d+1>d$.  Hence $\Gamma_b(T)>d$ and $T$ is non-diametrical.
\end{proof}

We next give
a sufficient condition, based on the number of vertices $|V|$ in $G$, for 
identifying non-diametrical graphs 
$G$.
\begin{lemma}\label{lemma:StrongMo}
Let $G= (V,E)$ be an arbitrary graph with diameter $\diam{G}$. Let $v_0$ be any vertex in $G$.  
Let $T$ be a breadth-first spanning tree rooted at $v_0$, and label the vertices in $T$ by the parity of their distances
from $v_0$ (0 at even distances and 1 at odd distances).  If the number of $1$'s in this labeling is greater than the diameter of 
$G$, then the graph $G$ is not diametrical. 
\end{lemma}

\begin{proof}
The number of $1$'s in the labeling of the spanning tree $T$ is the cost of a minimal dominating broadcast of $G$ with 
$f(v) =1$ for all $v \in V^+_f$.  Hence $\Gamma_b(G) \geq \Gamma(G) > |V^+_f|$ which is equal to the number of ones.
\end{proof}

\begin{corollary}\label{cor:StrongMo}
Let $G= (V,E)$ be an arbitrary graph with diameter $\diam{G}$.
 If $\left \lceil \frac{|V|}{2} \right \rceil \geq \diam{G}$, then $G$ is non-diametrical.
\end{corollary}

 \begin{example}
The following two trees both have $\Gamma_b(T)= \diam{T}+1$. Applying Lemma~\ref{lemma:StrongMo} shows that neither tree is 
diametrical.  It should be noted here that the spacing between the limbs for the first graph below satisfies the restrictions of Theorem~\ref{lemma:illegalspacing}.
It appears that both graphs satisfy the conditions given in Lemma~\ref{lemma:et} for a diametrical tree, however, upon closer scrutiny it is clear that a point is not diametrical and we see Lemma~\ref{lemma:et} does not apply.  That is, the graph on the right is non-diametrical and hence the the tree on the left is non-diametrical as well.
 
  
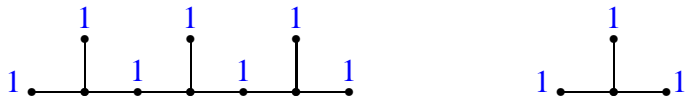
\begin{picture}(350,50)
\multiput(0,0)(20,0){7}{\circle*{3}}
\put(0,0){\textcolor{black}{\line(1,0){120}}}
 \put(-10,0){\textcolor{blue}{1}}
 \put(17,24){\textcolor{blue}{1}}
  \put(37,4){\textcolor{blue}{1}}
\put(57,24){\textcolor{blue}{1}}
\put(77,4){\textcolor{blue}{1}}
\put(97,24){\textcolor{blue}{1}}
\put(117,4){\textcolor{blue}{1}}
\put(20,0){\line(0,1){20}}
\put(60,0){\line(0,1){20}}
\put(100,0){\line(0,1){20}}
\multiput(20,0)(0,20){2}{\circle*{3}}
\multiput(60,0)(0,20){2}{\circle*{3}}
\multiput(100,0)(0,20){2}{\circle*{3}}

\multiput(200,0)(20,0){3}{\circle*{3}}
\put(200,0){\textcolor{black}{\line(1,0){40}}}
 \put(190,0){\textcolor{blue}{1}}
 \put(217,24){\textcolor{blue}{1}}
  \put(242,0){\textcolor{blue}{1}}
\put(220,0){\line(0,1){20}}
\multiput(220,0)(0,20){2}{\circle*{3}}

\end{picture}

 \end{example}

 We next consider how closely two legal limbs can be
 spaced on the diametrical path of a diametrical tree. 
 The list provided below shows all graphs containing two limbs which are placed too closely together to be diametrical. Note we only 
need to consider the spacing of the two legal limbs as Lemmas~\ref{lemma:Lobster} and \ref{firstlemma}  
eliminate the need to 
list 
any others.  Also we only provide cases where there are two vertices from the path $D$ that support protrusions.  
 It is assumed without loss of generality that the boxes attached to the ends of the  path shown are diametrical trees of 
diameters $d_1$ and $d_2$ respectively.  If the subgraphs shown are attached to non-diametrical trees, it is obvious that the 
resulting graph is non-diametrical.  Recall that a single vertex is non-diametrical.

\begin{lemma} \label{lemma:illegalspacing}
Let $T$ be a tree where the distance between two 
limbs satisfies one of the following inequalities. 
\begin{center}
\begin{tabular}{| c|c| c| } \hline 
$d(A,A) < 4$ & $d(A,B) < 3$ & 
$d(A,C) < 3$ \\ $d(B,B) < 3$ &
$d(B,C) < 2$ & 
$d(C,C) < 2$  \\ \hline
\end{tabular}
\end{center}
Then $T$ is not diametrical.  
\end{lemma}

\begin{proof}
To demonstrate that these graphs are not diametrical, it suffices to identify a single broadcast $f$ with $\cost{f}>\Gamma_b(G)$. We 
provide one such broadcast for each graph below.

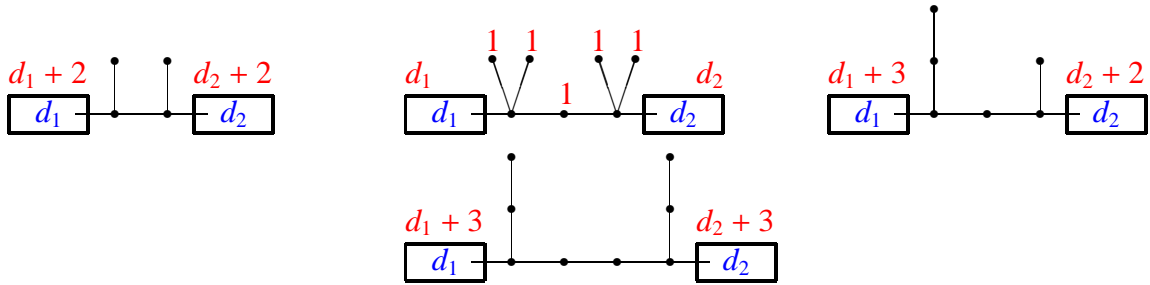
\vspace{5mm}
\begin{picture}(350,50)(0,-15)

\multiput(40,0)(20,0){2}{\circle*{3}}
\put(25,0){\textcolor{black}{\line(1,0){50}}}
 \put(10,-3){\textcolor{blue}{$d_1$}}
 \put(80,-3){\textcolor{blue}{$d_2$}}
 \put(0,12){\tcr{$d_1+2$}}
 \put(70,12){\tcr{$d_2+2$}}
   
\linethickness{1pt}
\put(0,-7){{\line(1,0){30}}}
\put(0,7){{\line(1,0){30}}}
\put(0,-7){{\line(0,1){14}}}
\put(30,-7){{\line(0,1){14}}}

\put(70,-7){{\line(1,0){30}}}
\put(70,7){{\line(1,0){30}}}
\put(70,-7){{\line(0,1){14}}}
\put(100,-7){{\line(0,1){14}}}

\linethickness{.1pt}
\put(40,0){\line(0,1){20}}
\put(40,20){\circle*{3}}
\put(60,0){\line(0,1){20}}
\put(60,20){\circle*{3}}

\multiput(190,0)(20,0){3}{\circle*{3}}
\put(175,0){\textcolor{black}{\line(1,0){70}}}
 \put(160,-3){\textcolor{blue}{$d_1$}}
 \put(250,-3){\textcolor{blue}{$d_2$}}
  \put(150,12){\tcr{$d_1$}}
 \put(260,12){\tcr{$d_2$}}
  \put(195,24){\tcr{$1$}}
   \put(180,24){\tcr{$1$}}
    \put(235,24){\tcr{$1$}}
   \put(220,24){\tcr{$1$}}
\put(208,4){\tcr{$1$}}
  
\linethickness{1pt}
\put(150,-7){{\line(1,0){30}}}
\put(150,7){{\line(1,0){30}}}
\put(150,-7){{\line(0,1){14}}}
\put(180,-7){{\line(0,1){14}}}

\put(240,-7){{\line(1,0){30}}}
\put(240,7){{\line(1,0){30}}}
\put(240,-7){{\line(0,1){14}}}
\put(270,-7){{\line(0,1){14}}}

\linethickness{.1pt}
\put(190,0){\line(1,3){7}}
\put(197,21){\circle*{3}}
\put(190,0){\line(-1,3){7}}
\put(183,21){\circle*{3}}

\put(230,0){\line(1,3){7}}
\put(237,21){\circle*{3}}
\put(230,0){\line(-1,3){7}}
\put(223,21){\circle*{3}}

\multiput(350,0)(20,0){3}{\circle*{3}}
\put(335,0){\textcolor{black}{\line(1,0){70}}}
 \put(320,-3){\textcolor{blue}{$d_1$}}
 \put(410,-3){\textcolor{blue}{$d_2$}}
 
   \put(310,12){\tcr{$d_1+3$}}
 \put(400,12){\tcr{$d_2+2$}}
 
\linethickness{1pt}
\put(310,-7){{\line(1,0){30}}}
\put(310,7){{\line(1,0){30}}}
\put(310,-7){{\line(0,1){14}}}
\put(340,-7){{\line(0,1){14}}}

\put(400,-7){{\line(1,0){30}}}
\put(400,7){{\line(1,0){30}}}
\put(400,-7){{\line(0,1){14}}}
\put(430,-7){{\line(0,1){14}}}

\linethickness{.1pt}
\put(350,0){\line(0,1){40}}
\put(350,20){\circle*{3}}
\put(350,40){\circle*{3}}

\put(390,0){\line(0,1){20}}
\put(390,20){\circle*{3}}
\end{picture}

\begin{picture}(350,50)(0,-10)
\multiput(190,0)(20,0){4}{\circle*{3}}
\put(175,0){\textcolor{black}{\line(1,0){90}}}
 \put(160,-3){\textcolor{blue}{$d_1$}}
 \put(270,-3){\textcolor{blue}{$d_2$}}
\put(150,12){\tcr{$d_1+3$}}
 \put(260,12){\tcr{$d_2+3$}}

\linethickness{1pt}
\put(150,-7){{\line(1,0){30}}}
\put(150,7){{\line(1,0){30}}}
\put(150,-7){{\line(0,1){14}}}
\put(180,-7){{\line(0,1){14}}}

\put(260,-7){{\line(1,0){30}}}
\put(260,7){{\line(1,0){30}}}
\put(260,-7){{\line(0,1){14}}}
\put(290,-7){{\line(0,1){14}}}

\linethickness{.1pt}
\put(190,0){\line(0,1){40}}
\put(190,20){\circle*{3}}
\put(190,40){\circle*{3}}

\put(250,0){\line(0,1){40}}
\put(250,20){\circle*{3}}
\put(250,40){\circle*{3}}

\end{picture}

 \begin{itemize}
 \item ($d(C,C) <2$, $d(B,C)<2$ ): For the first graph, there is a minimal broadcast $f$ with $\cost{f}=d+1$.  Set $f(v_0)=d_1+2$ and $f(v_d)=d_2+2$. Then 
$f$ 
is minimal because $y_{d_1,1}$ is a private neighbor of $v_0$ and $y_{d_2,1}$ is a private neighbor of $v_d$. The diameter of the 
tree is $d= d_1+d_2 +3$, and the cost of $f$ is $\cost{f} = d_1+d_2+4$.  Note that this also shows that $T$ is not diametrical
if $d(B,C)<2$.
 
 \item ($d(B,B) <3$): For the second graph, we set $f(v_0) = d_1$, $f(v_d)= d_2$, and $f(y_{d_1,1})=f(y_{d_1,2}) = f(v_{d_1+1}) = f(y_{d_i+2,1}) 
= f(y_{d_i+2,2}) =1$.  This broadcast is minimal because $v_{d_1-1}$ is a private neighbor of $v_0$, $v_{d_1+4}$ is a private 
neighbor of $v_d$, and the other 5 broadcasting nodes are their own private neighbors. Hence we see $d= d_1+d_2+4$ and $\cost{f} 
= 
d_1+d_2+5$.  
 
 \item ($d(A,C) <3$, $d(A,B)<3$): For the third graph, we set $f(v_0) = d_1+3$, $f(v_d)= d_2+2$. This broadcast is minimal as $y_{d_1,1,1}$ is a private 
neighbor of $v_0$ and $y_{d_1+2,1}$ is a private neighbor of $v_d$.  So while $d= d_1+d_2+4$,we have that $\cost{f} = d_1+d_2+5$. 
Note that this also shows that when $d(A,B)<3$ the tree is not diametrical.   
  
\item ($d(A,A) <4$): Similarly, on the fourth graph, we set $f(v_0) = d_1+3$, $f(v_d)= d_2+3$. This broadcast is minimal as $y_{d_1,1,1}$ is a 
private neighbor of $v_0$ and $y_{d_1+3,1,1}$ is a private neighbor of $v_d$.  So while $d= d_1+d_2+5$,we have that $\cost{f} = 
d_1+d_2+6$.  
 \end{itemize}\end{proof}

Note that for a graph to be non-diametrical, there must be a labelling with 
overlapping broadcasts, that is, the broadcast must be inefficient.  Because the graphs involved are trees, all overlaps can be 
detected by considering the interactions between the broadcasts covering two distinct protrusions.  Thus, to determine if a tree is 
diametrical or not, we need only look at how two protrusions interact.

To accomplish this, we show the smallest distance necessary between pairs of legal limbs on a diametrical path for the resulting tree to be diametrical.    What is considered to 
be sufficient distance varies depending on the type of protrusions involved.   To prove that the resulting graphs are 
diametrical, we show that the most costly, minimal broadcast on $G$ satisfies $\cost{f}=\Gamma_b(G)$, and then argue that no 
other 
minimal broadcast is more costly.

\begin{lemma}\label{lemma:limbdistances}
For a tree $T$ containing only legal limbs protruding from $D$ to be diametrical, 
the distances between its limbs must satisfy the following inequalities.

\begin{center}
\begin{tabular}{| c|c| c| } \hline 
$d(A,A) \geq 4$ & $d(A,B)  \geq 3$ & 
$d(A,C) \geq 3$ \\ $d(B,B)  \geq 3$ &
$d(B,C) \geq 2$ & 
$d(C,C) \geq 2$  \\ \hline
\end{tabular}
\end{center}

\end{lemma}

\begin{proof}
We first show that there are only a limited number of labelings 
(broadcasts) available for each type of protrusion.  In 
particular, if a single protrusion exists of length one there are only two possible 
labelings that produce a minimal 
broadcast.  If a protrusion is of degree two, there are only two possible labelings 
that produce a minimal broadcast.

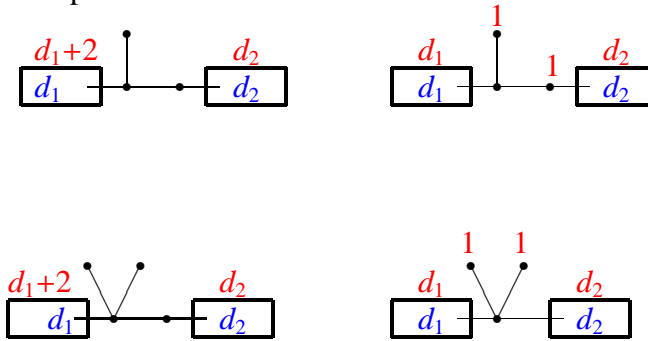
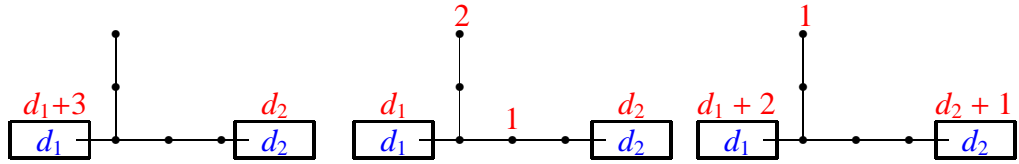
\begin{picture}(350,30)(-90,0)
\multiput(50,0)(20,0){2}{\circle*{3}}
\put(35,0){\textcolor{black}{\line(1,0){50}}}
 \put(15,-4){\textcolor{blue}{$d_1$}}
 \put(90,-4){\textcolor{blue}{$d_2$}}
 
\linethickness{1pt}
\put(10,-7){{\line(1,0){30}}}
\put(10,7){{\line(1,0){30}}}
\put(10,-7){{\line(0,1){14}}}
\put(40,-7){{\line(0,1){14}}}

\put(80,-7){{\line(1,0){30}}}
\put(80,7){{\line(1,0){30}}}
\put(80,-7){{\line(0,1){14}}}
\put(110,-7){{\line(0,1){14}}}
\put(15,10){\textcolor{red}{$d_1$+2}}
\put(90,10){\textcolor{red}{$d_2$}}

\linethickness{.1pt}
\put(50,0){\line(0,1){20}}
\put(50,20){\circle*{3}}
\multiput(190,0)(20,0){2}{\circle*{3}}
\put(175,0){\textcolor{black}{\line(1,0){50}}}
 \put(160,-4){\textcolor{blue}{$d_1$}}
 \put(230,-4){\textcolor{blue}{$d_2$}}
 
\linethickness{1pt}
\put(150,-7){{\line(1,0){30}}}
\put(150,7){{\line(1,0){30}}}
\put(150,-7){{\line(0,1){14}}}
\put(180,-7){{\line(0,1){14}}}

\put(220,-7){{\line(1,0){30}}}
\put(220,7){{\line(1,0){30}}}
\put(220,-7){{\line(0,1){14}}}
\put(250,-7){{\line(0,1){14}}}

\linethickness{.1pt}
\put(190,0){\line(0,1){20}}
\put(190,20){\circle*{3}}
\put(187,23){\textcolor{red}{1}}
\put(208,4){\textcolor{red}{1}}
\put(160,10){\textcolor{red}{$d_1$}}
\put(230,10){\textcolor{red}{$d_2$}}
\end{picture}

\vspace{20mm}
\begin{picture}(350,30)(-90,0)
\multiput(45,0)(20,0){2}{\circle*{3}}
\put(30,0){\textcolor{black}{\line(1,0){50}}}
 \put(20,-4){\textcolor{blue}{$d_1$}}
 \put(85,-4){\textcolor{blue}{$d_2$}}
 
\linethickness{1pt}
\put(5,-7){{\line(1,0){30}}}
\put(5,7){{\line(1,0){30}}}
\put(5,-7){{\line(0,1){14}}}
\put(35,-7){{\line(0,1){14}}}

\put(75,-7){{\line(1,0){30}}}
\put(75,7){{\line(1,0){30}}}
\put(75,-7){{\line(0,1){14}}}
\put(105,-7){{\line(0,1){14}}}
\put(5,10){\textcolor{red}{$d_1$+2}}
\put(85,10){\textcolor{red}{$d_2$}}

\linethickness{.1pt}
\put(45,0){\line(-1,2){10}}
\put(45,0){\line(1,2){10}}
\put(35,20){\circle*{3}}
\put(55,20){\circle*{3}}
\multiput(190,0)(20,0){1}{\circle*{3}}
\put(175,0){\textcolor{black}{\line(1,0){40}}}
 \put(160,-4){\textcolor{blue}{$d_1$}}
 \put(220,-4){\textcolor{blue}{$d_2$}}
 
\linethickness{1pt}
\put(150,-7){{\line(1,0){30}}}
\put(150,7){{\line(1,0){30}}}
\put(150,-7){{\line(0,1){14}}}
\put(180,-7){{\line(0,1){14}}}

\put(210,-7){{\line(1,0){30}}}
\put(210,7){{\line(1,0){30}}}
\put(210,-7){{\line(0,1){14}}}
\put(240,-7){{\line(0,1){14}}}

\linethickness{.1pt}
\put(190,0){\line(-1,2){10}}
\put(190,0){\line(1,2){10}}
\put(180,20){\circle*{3}}
\put(200,20){\circle*{3}}
\put(176,25){\textcolor{red}{1}}
\put(196,25){\textcolor{red}{1}}
\put(160,10){\textcolor{red}{$d_1$}}
\put(220,10){\textcolor{red}{$d_2$}}
\end{picture}

\vspace{20mm}
\begin{picture}(350,30)(-90,5)
\multiput(0,0)(20,0){3}{\circle*{3}}
\put(-15,0){\textcolor{black}{\line(1,0){65}}}
 \put(-30,-4){\textcolor{blue}{$d_1$}}
 \put(55,-4){\textcolor{blue}{$d_2$}}
 
\linethickness{1pt}
\put(-40,-7){{\line(1,0){30}}}
\put(-40,7){{\line(1,0){30}}}
\put(-40,-7){{\line(0,1){14}}}
\put(-10,-7){{\line(0,1){14}}}

\put(45,-7){{\line(1,0){30}}}
\put(45,7){{\line(1,0){30}}}
\put(45,-7){{\line(0,1){14}}}
\put(75,-7){{\line(0,1){14}}}
\put(-35,10){\textcolor{red}{$d_1$+3}}
\put(55,10){\textcolor{red}{$d_2$}}

\linethickness{.1pt}
\put(0,0){\line(0,1){40}}
\multiput(0,20)(0,20){2}{\circle*{3}}
\multiput(130,0)(20,0){3}{\circle*{3}}
\put(115,0){\textcolor{black}{\line(1,0){70}}}
 \put(100,-4){\textcolor{blue}{$d_1$}}
 \put(190,-4){\textcolor{blue}{$d_2$}}
 
\linethickness{1pt}
\put(90,-7){{\line(1,0){30}}}
\put(90,7){{\line(1,0){30}}}
\put(90,-7){{\line(0,1){14}}}
\put(120,-7){{\line(0,1){14}}}

\put(180,-7){{\line(1,0){30}}}
\put(180,7){{\line(1,0){30}}}
\put(180,-7){{\line(0,1){14}}}
\put(210,-7){{\line(0,1){14}}}

\linethickness{.1pt}
\put(130,0){\line(0,1){40}}
\multiput(130,20)(0,20){2}{\circle*{3}}
\put(128,43){\textcolor{red}{2}}
\put(147,4){\textcolor{red}{1}}
\put(100,10){\textcolor{red}{$d_1$}}
\put(190,10){\textcolor{red}{$d_2$}}
\multiput(260,0)(20,0){3}{\circle*{3}}
\put(245,0){\textcolor{black}{\line(1,0){70}}}
 \put(230,-4){\textcolor{blue}{$d_1$}}
 \put(320,-4){\textcolor{blue}{$d_2$}}
 
 \linethickness{1pt}
\put(220,-7){{\line(1,0){30}}}
\put(220,7){{\line(1,0){30}}}
\put(220,-7){{\line(0,1){14}}}
\put(250,-7){{\line(0,1){14}}}

\put(310,-7){{\line(1,0){30}}}
\put(310,7){{\line(1,0){30}}}
\put(310,-7){{\line(0,1){14}}}
\put(340,-7){{\line(0,1){14}}}

\linethickness{.1pt}
\put(260,0){\line(0,1){40}}
\multiput(260,20)(0,20){2}{\circle*{3}}
\put(258,43){\textcolor{red}{1}}
\put(220,10){\textcolor{red}{$d_1+2$}}
\put(310,10){\textcolor{red}{$d_2+1$}}
\end{picture}
\vspace{5mm}

Next we consider the situations wherein a combination of two of the graphs above 
(with attached labelings) are glued together.

Note that the list of illegal subgraphs shows that the two protrusions must be at least distance one, two or three apart for the 
graph to be diametrical depending on the types of protrusions.
Below we consider all possible combinations of two legal limbs placed as close as possible together.  It is evident that the 
(broadcast) labelings are diametrical.  Furthermore, including additional edges between the two protrusions also gives a 
diametrical tree.

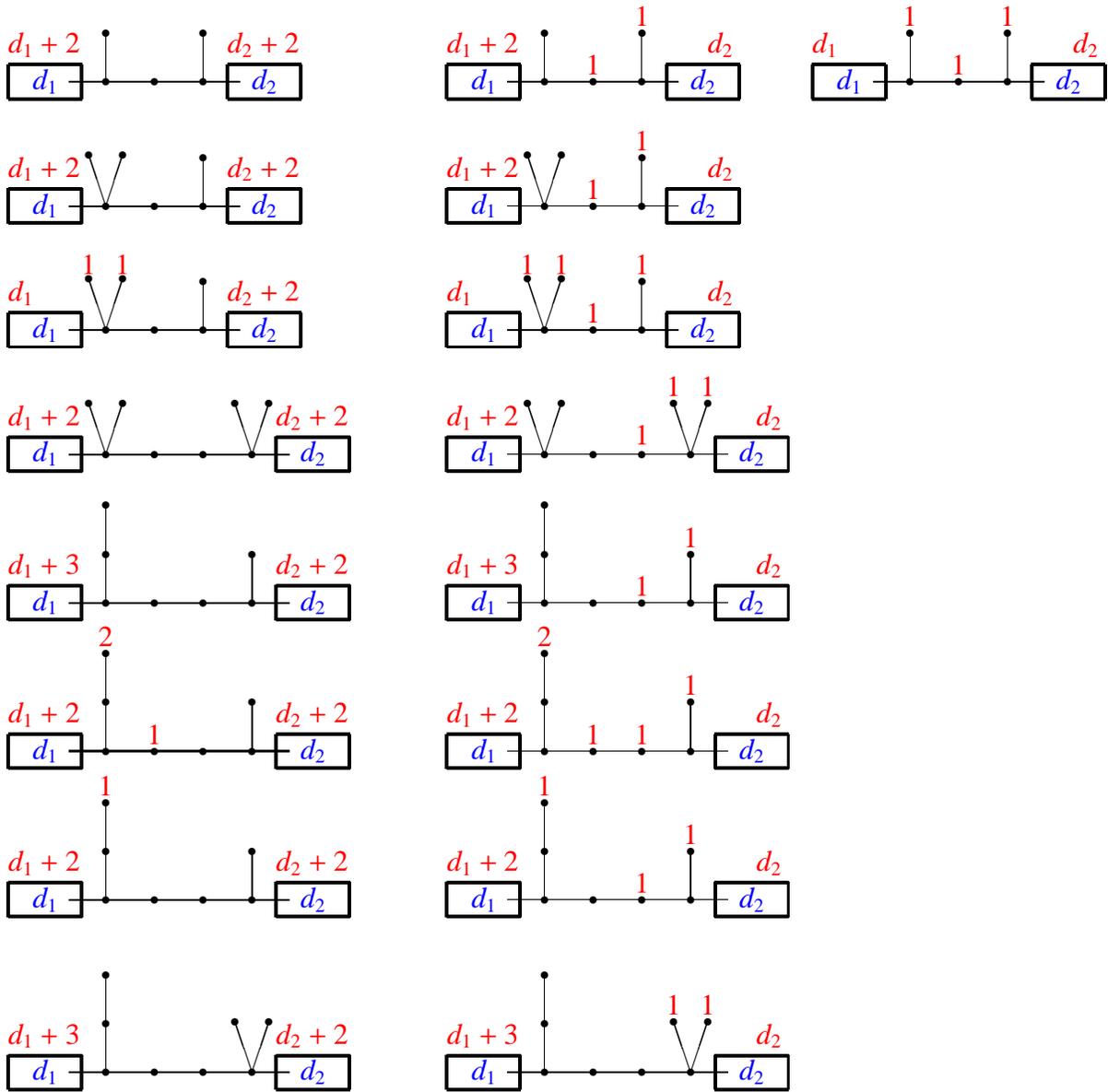
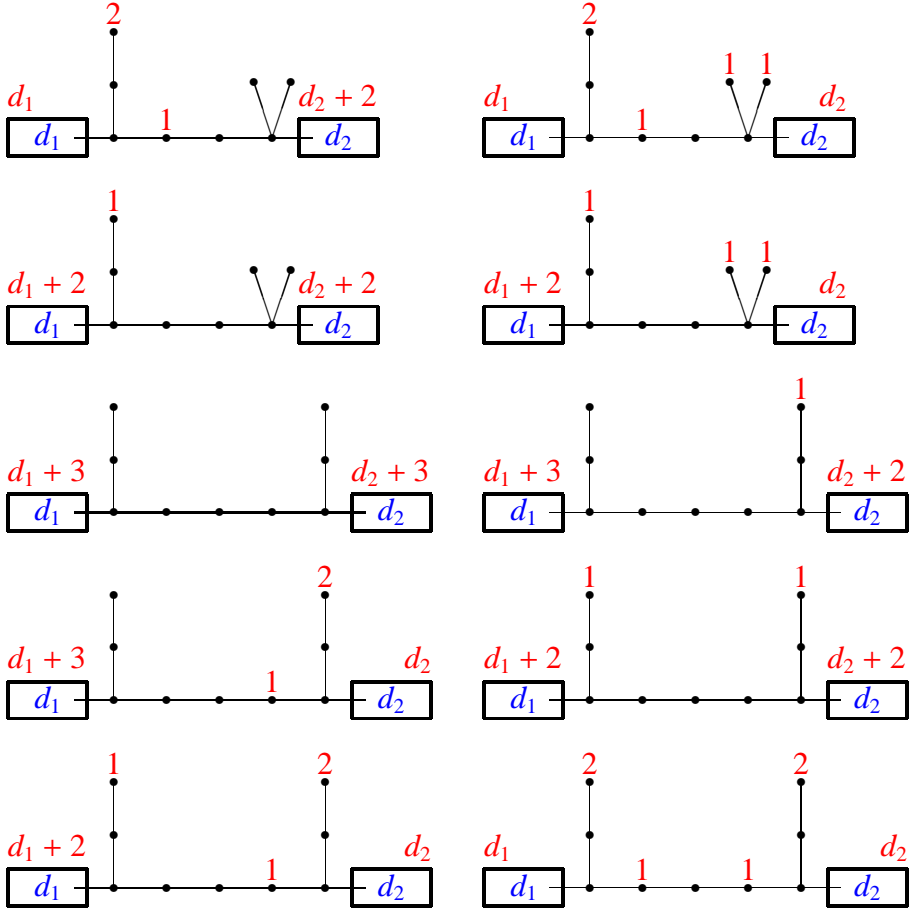
\begin{picture}(350,70)(0,0)
\multiput(40,0)(20,0){3}{\circle*{3}}
\put(25,0){\textcolor{black}{\line(1,0){70}}}
\put(10,-3){\textcolor{blue}{$d_1$}}
\put(100,-3){\textcolor{blue}{$d_2$}}
\put(0,12){\tcr{$d_1+2$}}
\put(90,12){\tcr{$d_2+2$}}

\linethickness{1pt}
\put(0,-7){{\line(1,0){30}}}
\put(0,7){{\line(1,0){30}}}
\put(0,-7){{\line(0,1){14}}}
\put(30,-7){{\line(0,1){14}}}

\put(90,-7){{\line(1,0){30}}}
\put(90,7){{\line(1,0){30}}}
\put(90,-7){{\line(0,1){14}}}
\put(120,-7){{\line(0,1){14}}}

\linethickness{.1pt}
\put(40,0){\line(0,1){20}}
\put(40,20){\circle*{3}}

\put(80,0){\line(0,1){20}}
\put(80,20){\circle*{3}}

\multiput(220,0)(20,0){3}{\circle*{3}}
\put(205,0){\textcolor{black}{\line(1,0){70}}}
\put(190,-3){\textcolor{blue}{$d_1$}}
\put(280,-3){\textcolor{blue}{$d_2$}}
 
\put(180,12){\tcr{$d_1+2$}}

\put(287,12){\tcr{$d_2$}}
\put(257,23){\tcr{$1$}}
\put(237,03){\tcr{$1$}}

\linethickness{1pt}
\put(180,-7){{\line(1,0){30}}}
\put(180,7){{\line(1,0){30}}}
\put(180,-7){{\line(0,1){14}}}
\put(210,-7){{\line(0,1){14}}}

\put(270,-7){{\line(1,0){30}}}
\put(270,7){{\line(1,0){30}}}
\put(270,-7){{\line(0,1){14}}}
\put(300,-7){{\line(0,1){14}}}

\linethickness{.1pt}
\put(220,0){\line(0,1){20}}
\put(220,20){\circle*{3}}

\put(260,0){\line(0,1){20}}
\put(260,20){\circle*{3}}

\multiput(370,0)(20,0){3}{\circle*{3}}
\put(355,0){\textcolor{black}{\line(1,0){70}}}
\put(340,-3){\textcolor{blue}{$d_1$}}
\put(430,-3){\textcolor{blue}{$d_2$}}
 
\put(330,12){\tcr{$d_1$}}
\put(367,23){\tcr{$1$}}

\put(437,12){\tcr{$d_2$}}
\put(407,23){\tcr{$1$}}
\put(387,03){\tcr{$1$}}

\linethickness{1pt}
\put(330,-7){{\line(1,0){30}}}
\put(330,7){{\line(1,0){30}}}
\put(330,-7){{\line(0,1){14}}}
\put(360,-7){{\line(0,1){14}}}

\put(420,-7){{\line(1,0){30}}}
\put(420,7){{\line(1,0){30}}}
\put(420,-7){{\line(0,1){14}}}
\put(450,-7){{\line(0,1){14}}}

\linethickness{.1pt}
\put(370,0){\line(0,1){20}}
\put(370,20){\circle*{3}}

\put(410,0){\line(0,1){20}}
\put(410,20){\circle*{3}}

\end{picture}

\begin{picture}(350,50)(0,0)
\multiput(40,0)(20,0){3}{\circle*{3}}
\put(25,0){\textcolor{black}{\line(1,0){70}}}
\put(10,-3){\textcolor{blue}{$d_1$}}
\put(100,-3){\textcolor{blue}{$d_2$}}
\put(0,12){\tcr{$d_1+2$}}
\put(90,12){\tcr{$d_2+2$}}

\linethickness{1pt}
\put(0,-7){{\line(1,0){30}}}
\put(0,7){{\line(1,0){30}}}
\put(0,-7){{\line(0,1){14}}}
\put(30,-7){{\line(0,1){14}}}

\put(90,-7){{\line(1,0){30}}}
\put(90,7){{\line(1,0){30}}}
\put(90,-7){{\line(0,1){14}}}
\put(120,-7){{\line(0,1){14}}}

\linethickness{.1pt}
\put(40,0){\line(-1,3){7}}
\put(33,21){\circle*{3}}
\put(40,0){\line(1,3){7}}
\put(47,21){\circle*{3}}

\put(80,0){\line(0,1){20}}
\put(80,20){\circle*{3}}

\multiput(220,0)(20,0){3}{\circle*{3}}
\put(205,0){\textcolor{black}{\line(1,0){70}}}
\put(190,-3){\textcolor{blue}{$d_1$}}
\put(280,-3){\textcolor{blue}{$d_2$}}
 
\put(180,12){\tcr{$d_1+2$}}

\put(287,12){\tcr{$d_2$}}
\put(257,23){\tcr{$1$}}
\put(237,03){\tcr{$1$}}

\linethickness{1pt}
\put(180,-7){{\line(1,0){30}}}
\put(180,7){{\line(1,0){30}}}
\put(180,-7){{\line(0,1){14}}}
\put(210,-7){{\line(0,1){14}}}

\put(270,-7){{\line(1,0){30}}}
\put(270,7){{\line(1,0){30}}}
\put(270,-7){{\line(0,1){14}}}
\put(300,-7){{\line(0,1){14}}}

\linethickness{.1pt}
\put(220,0){\line(-1,3){7}}
\put(213,21){\circle*{3}}
\put(220,0){\line(1,3){7}}
\put(227,21){\circle*{3}}

\put(260,0){\line(0,1){20}}
\put(260,20){\circle*{3}}

\end{picture}

\begin{picture}(350,50)(0,0)
\multiput(40,0)(20,0){3}{\circle*{3}}
\put(25,0){\textcolor{black}{\line(1,0){70}}}
\put(10,-3){\textcolor{blue}{$d_1$}}
\put(100,-3){\textcolor{blue}{$d_2$}}
\put(0,12){\tcr{$d_1$}}
\put(90,12){\tcr{$d_2+2$}}

\put(30,23){\tcr{$1$}}
\put(44,23){\tcr{$1$}}

\linethickness{1pt}
\put(0,-7){{\line(1,0){30}}}
\put(0,7){{\line(1,0){30}}}
\put(0,-7){{\line(0,1){14}}}
\put(30,-7){{\line(0,1){14}}}

\put(90,-7){{\line(1,0){30}}}
\put(90,7){{\line(1,0){30}}}
\put(90,-7){{\line(0,1){14}}}
\put(120,-7){{\line(0,1){14}}}

\linethickness{.1pt}
\put(40,0){\line(-1,3){7}}
\put(33,21){\circle*{3}}
\put(40,0){\line(1,3){7}}
\put(47,21){\circle*{3}}

\put(80,0){\line(0,1){20}}
\put(80,20){\circle*{3}}

\multiput(220,0)(20,0){3}{\circle*{3}}
\put(205,0){\textcolor{black}{\line(1,0){70}}}
\put(190,-3){\textcolor{blue}{$d_1$}}
\put(280,-3){\textcolor{blue}{$d_2$}}
 
\put(180,12){\tcr{$d_1$}}
\put(210,23){\tcr{$1$}}
\put(224,23){\tcr{$1$}}

\put(287,12){\tcr{$d_2$}}
\put(257,23){\tcr{$1$}}
\put(237,03){\tcr{$1$}}

\linethickness{1pt}
\put(180,-7){{\line(1,0){30}}}
\put(180,7){{\line(1,0){30}}}
\put(180,-7){{\line(0,1){14}}}
\put(210,-7){{\line(0,1){14}}}

\put(270,-7){{\line(1,0){30}}}
\put(270,7){{\line(1,0){30}}}
\put(270,-7){{\line(0,1){14}}}
\put(300,-7){{\line(0,1){14}}}

\linethickness{.1pt}
\put(220,0){\line(-1,3){7}}
\put(213,21){\circle*{3}}
\put(220,0){\line(1,3){7}}
\put(227,21){\circle*{3}}

\put(260,0){\line(0,1){20}}
\put(260,20){\circle*{3}}

\end{picture}

\begin{picture}(350,50)(0,0)
\multiput(40,0)(20,0){4}{\circle*{3}}
\put(25,0){\textcolor{black}{\line(1,0){90}}}
 \put(10,-3){\textcolor{blue}{$d_1$}}
 \put(120,-3){\textcolor{blue}{$d_2$}}
 \put(0,12){\tcr{$d_1+2$}}
 \put(110,12){\tcr{$d_2+2$}}

\linethickness{1pt}
\put(0,-7){{\line(1,0){30}}}
\put(0,7){{\line(1,0){30}}}
\put(0,-7){{\line(0,1){14}}}
\put(30,-7){{\line(0,1){14}}}

\put(110,-7){{\line(1,0){30}}}
\put(110,7){{\line(1,0){30}}}
\put(110,-7){{\line(0,1){14}}}
\put(140,-7){{\line(0,1){14}}}

\linethickness{.1pt}
\put(40,0){\line(-1,3){7}}
\put(33,21){\circle*{3}}
\put(40,0){\line(1,3){7}}
\put(47,21){\circle*{3}}

\put(100,0){\line(-1,3){7}}
\put(93,21){\circle*{3}}
\put(100,0){\line(1,3){7}}
\put(107,21){\circle*{3}}

\multiput(220,0)(20,0){4}{\circle*{3}}
\put(205,0){\textcolor{black}{\line(1,0){90}}}
 \put(190,-3){\textcolor{blue}{$d_1$}}
 \put(300,-3){\textcolor{blue}{$d_2$}}
 
 \put(180,12){\tcr{$d_1+2$}}

 \put(307,12){\tcr{$d_2$}}
 \put(270,24){\tcr{$1$}}
\put(284,24){\tcr{$1$}}
\put(257,04){\tcr{$1$}}

\linethickness{1pt}
\put(180,-7){{\line(1,0){30}}}
\put(180,7){{\line(1,0){30}}}
\put(180,-7){{\line(0,1){14}}}
\put(210,-7){{\line(0,1){14}}}

\put(290,-7){{\line(1,0){30}}}
\put(290,7){{\line(1,0){30}}}
\put(290,-7){{\line(0,1){14}}}
\put(320,-7){{\line(0,1){14}}}

\linethickness{.1pt}
\put(220,0){\line(-1,3){7}}
\put(213,21){\circle*{3}}
\put(220,0){\line(1,3){7}}
\put(227,21){\circle*{3}}

\put(280,0){\line(-1,3){7}}
\put(273,21){\circle*{3}}
\put(280,0){\line(1,3){7}}
\put(287,21){\circle*{3}}
\end{picture}

\begin{picture}(350,60)(0,0)
\multiput(40,0)(20,0){4}{\circle*{3}}
\put(25,0){\textcolor{black}{\line(1,0){90}}}
 \put(10,-3){\textcolor{blue}{$d_1$}}
 \put(120,-3){\textcolor{blue}{$d_2$}}
 \put(0,12){\tcr{$d_1+3$}}
 \put(110,12){\tcr{$d_2+2$}}

\linethickness{1pt}
\put(0,-7){{\line(1,0){30}}}
\put(0,7){{\line(1,0){30}}}
\put(0,-7){{\line(0,1){14}}}
\put(30,-7){{\line(0,1){14}}}

\put(110,-7){{\line(1,0){30}}}
\put(110,7){{\line(1,0){30}}}
\put(110,-7){{\line(0,1){14}}}
\put(140,-7){{\line(0,1){14}}}

\linethickness{.1pt}
\put(40,0){\line(0,1){40}}
\put(40,20){\circle*{3}}
\put(40,40){\circle*{3}}

\put(100,0){\line(0,1){20}}
\put(100,20){\circle*{3}}

\multiput(220,0)(20,0){4}{\circle*{3}}
\put(205,0){\textcolor{black}{\line(1,0){90}}}
 \put(190,-3){\textcolor{blue}{$d_1$}}
 \put(300,-3){\textcolor{blue}{$d_2$}}
 
 \put(180,12){\tcr{$d_1+3$}}
  
 \put(307,12){\tcr{$d_2$}}
 \put(277,23){\tcr{$1$}}
\put(257,3){\tcr{$1$}}

\linethickness{1pt}
\put(180,-7){{\line(1,0){30}}}
\put(180,7){{\line(1,0){30}}}
\put(180,-7){{\line(0,1){14}}}
\put(210,-7){{\line(0,1){14}}}

\put(290,-7){{\line(1,0){30}}}
\put(290,7){{\line(1,0){30}}}
\put(290,-7){{\line(0,1){14}}}
\put(320,-7){{\line(0,1){14}}}

\linethickness{.1pt}
\put(220,0){\line(0,1){40}}
\put(220,20){\circle*{3}}
\put(220,40){\circle*{3}}

\put(280,0){\line(0,1){20}}
\put(280,20){\circle*{3}}
\end{picture}

\begin{picture}(350,60)(0,0)
\multiput(40,0)(20,0){4}{\circle*{3}}
\put(25,0){\textcolor{black}{\line(1,0){90}}}
 \put(10,-3){\textcolor{blue}{$d_1$}}
 \put(120,-3){\textcolor{blue}{$d_2$}}
 \put(0,12){\tcr{$d_1+2$}}
 \put(110,12){\tcr{$d_2+2$}}
 \put(37,43){\tcr{$2$}}
 \put(57,03){\tcr{$1$}}

\linethickness{1pt}
\put(0,-7){{\line(1,0){30}}}
\put(0,7){{\line(1,0){30}}}
\put(0,-7){{\line(0,1){14}}}
\put(30,-7){{\line(0,1){14}}}

\put(110,-7){{\line(1,0){30}}}
\put(110,7){{\line(1,0){30}}}
\put(110,-7){{\line(0,1){14}}}
\put(140,-7){{\line(0,1){14}}}

\linethickness{.1pt}
\put(40,0){\line(0,1){40}}
\put(40,20){\circle*{3}}
\put(40,40){\circle*{3}}

\put(100,0){\line(0,1){20}}
\put(100,20){\circle*{3}}

\multiput(220,0)(20,0){4}{\circle*{3}}
\put(205,0){\textcolor{black}{\line(1,0){90}}}
 \put(190,-3){\textcolor{blue}{$d_1$}}
 \put(300,-3){\textcolor{blue}{$d_2$}}
 
 \put(180,12){\tcr{$d_1+2$}}
 \put(217,43){\tcr{$2$}}
 \put(237,03){\tcr{$1$}}
  
 \put(307,12){\tcr{$d_2$}}
 \put(277,23){\tcr{$1$}}
\put(257,3){\tcr{$1$}}

\linethickness{1pt}
\put(180,-7){{\line(1,0){30}}}
\put(180,7){{\line(1,0){30}}}
\put(180,-7){{\line(0,1){14}}}
\put(210,-7){{\line(0,1){14}}}

\put(290,-7){{\line(1,0){30}}}
\put(290,7){{\line(1,0){30}}}
\put(290,-7){{\line(0,1){14}}}
\put(320,-7){{\line(0,1){14}}}

\linethickness{.1pt}
\put(220,0){\line(0,1){40}}
\put(220,20){\circle*{3}}
\put(220,40){\circle*{3}}

\put(280,0){\line(0,1){20}}
\put(280,20){\circle*{3}}
\end{picture}

\begin{picture}(350,60)(0,0)
\multiput(40,0)(20,0){4}{\circle*{3}}
\put(25,0){\textcolor{black}{\line(1,0){90}}}
 \put(10,-3){\textcolor{blue}{$d_1$}}
 \put(120,-3){\textcolor{blue}{$d_2$}}
 \put(0,12){\tcr{$d_1+2$}}
 \put(110,12){\tcr{$d_2+2$}}
 \put(37,43){\tcr{$1$}}

\linethickness{1pt}
\put(0,-7){{\line(1,0){30}}}
\put(0,7){{\line(1,0){30}}}
\put(0,-7){{\line(0,1){14}}}
\put(30,-7){{\line(0,1){14}}}

\put(110,-7){{\line(1,0){30}}}
\put(110,7){{\line(1,0){30}}}
\put(110,-7){{\line(0,1){14}}}
\put(140,-7){{\line(0,1){14}}}

\linethickness{.1pt}
\put(40,0){\line(0,1){40}}
\put(40,20){\circle*{3}}
\put(40,40){\circle*{3}}

\put(100,0){\line(0,1){20}}
\put(100,20){\circle*{3}}

\multiput(220,0)(20,0){4}{\circle*{3}}
\put(205,0){\textcolor{black}{\line(1,0){90}}}
 \put(190,-3){\textcolor{blue}{$d_1$}}
 \put(300,-3){\textcolor{blue}{$d_2$}}
 
 \put(180,12){\tcr{$d_1+2$}}
 \put(217,43){\tcr{$1$}}

 \put(307,12){\tcr{$d_2$}}
 \put(277,23){\tcr{$1$}}
\put(257,3){\tcr{$1$}}

\linethickness{1pt}
\put(180,-7){{\line(1,0){30}}}
\put(180,7){{\line(1,0){30}}}
\put(180,-7){{\line(0,1){14}}}
\put(210,-7){{\line(0,1){14}}}

\put(290,-7){{\line(1,0){30}}}
\put(290,7){{\line(1,0){30}}}
\put(290,-7){{\line(0,1){14}}}
\put(320,-7){{\line(0,1){14}}}

\linethickness{.1pt}
\put(220,0){\line(0,1){40}}
\put(220,20){\circle*{3}}
\put(220,40){\circle*{3}}

\put(280,0){\line(0,1){20}}
\put(280,20){\circle*{3}}
\end{picture}

\begin{picture}(350,70)(0,0)
\multiput(40,0)(20,0){4}{\circle*{3}}
\put(25,0){\textcolor{black}{\line(1,0){90}}}
 \put(10,-3){\textcolor{blue}{$d_1$}}
 \put(120,-3){\textcolor{blue}{$d_2$}}
 \put(0,12){\tcr{$d_1+3$}}
 \put(110,12){\tcr{$d_2+2$}}

\linethickness{1pt}
\put(0,-7){{\line(1,0){30}}}
\put(0,7){{\line(1,0){30}}}
\put(0,-7){{\line(0,1){14}}}
\put(30,-7){{\line(0,1){14}}}

\put(110,-7){{\line(1,0){30}}}
\put(110,7){{\line(1,0){30}}}
\put(110,-7){{\line(0,1){14}}}
\put(140,-7){{\line(0,1){14}}}

\linethickness{.1pt}
\put(40,0){\line(0,1){40}}
\put(40,20){\circle*{3}}
\put(40,40){\circle*{3}}

\put(100,0){\line(-1,3){7}}
\put(93,21){\circle*{3}}
\put(100,0){\line(1,3){7}}
\put(107,21){\circle*{3}}

\multiput(220,0)(20,0){4}{\circle*{3}}
\put(205,0){\textcolor{black}{\line(1,0){90}}}
 \put(190,-3){\textcolor{blue}{$d_1$}}
 \put(300,-3){\textcolor{blue}{$d_2$}}
 
 \put(180,12){\tcr{$d_1+3$}}

 \put(307,12){\tcr{$d_2$}}
 \put(270,24){\tcr{$1$}}
\put(284,24){\tcr{$1$}}

\linethickness{1pt}
\put(180,-7){{\line(1,0){30}}}
\put(180,7){{\line(1,0){30}}}
\put(180,-7){{\line(0,1){14}}}
\put(210,-7){{\line(0,1){14}}}

\put(290,-7){{\line(1,0){30}}}
\put(290,7){{\line(1,0){30}}}
\put(290,-7){{\line(0,1){14}}}
\put(320,-7){{\line(0,1){14}}}

\linethickness{.1pt}
\put(220,0){\line(0,1){40}}
\put(220,20){\circle*{3}}
\put(220,40){\circle*{3}}

\put(280,0){\line(-1,3){7}}
\put(273,21){\circle*{3}}
\put(280,0){\line(1,3){7}}
\put(287,21){\circle*{3}}
\end{picture}

\begin{picture}(350,70)(0,0)
\multiput(40,0)(20,0){4}{\circle*{3}}
\put(25,0){\textcolor{black}{\line(1,0){90}}}
 \put(10,-3){\textcolor{blue}{$d_1$}}
 \put(120,-3){\textcolor{blue}{$d_2$}}
 \put(0,12){\tcr{$d_1$}}
 \put(110,12){\tcr{$d_2+2$}}
\put(37,43){\tcr{$2$}}
 \put(57,03){\tcr{$1$}}

\linethickness{1pt}
\put(0,-7){{\line(1,0){30}}}
\put(0,7){{\line(1,0){30}}}
\put(0,-7){{\line(0,1){14}}}
\put(30,-7){{\line(0,1){14}}}

\put(110,-7){{\line(1,0){30}}}
\put(110,7){{\line(1,0){30}}}
\put(110,-7){{\line(0,1){14}}}
\put(140,-7){{\line(0,1){14}}}

\linethickness{.1pt}
\put(40,0){\line(0,1){40}}
\put(40,20){\circle*{3}}
\put(40,40){\circle*{3}}

\put(100,0){\line(-1,3){7}}
\put(93,21){\circle*{3}}
\put(100,0){\line(1,3){7}}
\put(107,21){\circle*{3}}

\multiput(220,0)(20,0){4}{\circle*{3}}
\put(205,0){\textcolor{black}{\line(1,0){90}}}
 \put(190,-3){\textcolor{blue}{$d_1$}}
 \put(300,-3){\textcolor{blue}{$d_2$}}
 
 \put(180,12){\tcr{$d_1$}}

 \put(217,43){\tcr{$2$}}
 \put(237,03){\tcr{$1$}}

 \put(307,12){\tcr{$d_2$}}
 \put(270,24){\tcr{$1$}}
\put(284,24){\tcr{$1$}}

\linethickness{1pt}
\put(180,-7){{\line(1,0){30}}}
\put(180,7){{\line(1,0){30}}}
\put(180,-7){{\line(0,1){14}}}
\put(210,-7){{\line(0,1){14}}}

\put(290,-7){{\line(1,0){30}}}
\put(290,7){{\line(1,0){30}}}
\put(290,-7){{\line(0,1){14}}}
\put(320,-7){{\line(0,1){14}}}

\linethickness{.1pt}
\put(220,0){\line(0,1){40}}
\put(220,20){\circle*{3}}
\put(220,40){\circle*{3}}

\put(280,0){\line(-1,3){7}}
\put(273,21){\circle*{3}}
\put(280,0){\line(1,3){7}}
\put(287,21){\circle*{3}}
\end{picture}

\begin{picture}(350,70)(0,0)
\multiput(40,0)(20,0){4}{\circle*{3}}
\put(25,0){\textcolor{black}{\line(1,0){90}}}
 \put(10,-3){\textcolor{blue}{$d_1$}}
 \put(120,-3){\textcolor{blue}{$d_2$}}
 \put(0,12){\tcr{$d_1+2$}}
 \put(110,12){\tcr{$d_2+2$}}
\put(37,43){\tcr{$1$}}

\linethickness{1pt}
\put(0,-7){{\line(1,0){30}}}
\put(0,7){{\line(1,0){30}}}
\put(0,-7){{\line(0,1){14}}}
\put(30,-7){{\line(0,1){14}}}

\put(110,-7){{\line(1,0){30}}}
\put(110,7){{\line(1,0){30}}}
\put(110,-7){{\line(0,1){14}}}
\put(140,-7){{\line(0,1){14}}}

\linethickness{.1pt}
\put(40,0){\line(0,1){40}}
\put(40,20){\circle*{3}}
\put(40,40){\circle*{3}}

\put(100,0){\line(-1,3){7}}
\put(93,21){\circle*{3}}
\put(100,0){\line(1,3){7}}
\put(107,21){\circle*{3}}

\multiput(220,0)(20,0){4}{\circle*{3}}
\put(205,0){\textcolor{black}{\line(1,0){90}}}
 \put(190,-3){\textcolor{blue}{$d_1$}}
 \put(300,-3){\textcolor{blue}{$d_2$}}
 
 \put(180,12){\tcr{$d_1+2$}}
 \put(217,43){\tcr{$1$}}

 \put(307,12){\tcr{$d_2$}}
 \put(270,24){\tcr{$1$}}
\put(284,24){\tcr{$1$}}

\linethickness{1pt}
\put(180,-7){{\line(1,0){30}}}
\put(180,7){{\line(1,0){30}}}
\put(180,-7){{\line(0,1){14}}}
\put(210,-7){{\line(0,1){14}}}

\put(290,-7){{\line(1,0){30}}}
\put(290,7){{\line(1,0){30}}}
\put(290,-7){{\line(0,1){14}}}
\put(320,-7){{\line(0,1){14}}}

\linethickness{.1pt}
\put(220,0){\line(0,1){40}}
\put(220,20){\circle*{3}}
\put(220,40){\circle*{3}}

\put(280,0){\line(-1,3){7}}
\put(273,21){\circle*{3}}
\put(280,0){\line(1,3){7}}
\put(287,21){\circle*{3}}
\end{picture}

\begin{picture}(350,70)(0,0)
\multiput(40,0)(20,0){5}{\circle*{3}}
\put(25,0){\textcolor{black}{\line(1,0){110}}}
 \put(10,-3){\textcolor{blue}{$d_1$}}
 \put(140,-3){\textcolor{blue}{$d_2$}}
 \put(0,12){\tcr{$d_1+3$}}
 \put(130,12){\tcr{$d_2+3$}}

\linethickness{1pt}
\put(0,-7){{\line(1,0){30}}}
\put(0,7){{\line(1,0){30}}}
\put(0,-7){{\line(0,1){14}}}
\put(30,-7){{\line(0,1){14}}}

\put(130,-7){{\line(1,0){30}}}
\put(130,7){{\line(1,0){30}}}
\put(130,-7){{\line(0,1){14}}}
\put(160,-7){{\line(0,1){14}}}

\linethickness{.1pt}
\put(40,0){\line(0,1){40}}
\put(40,20){\circle*{3}}
\put(40,40){\circle*{3}}

\put(120,0){\line(0,1){40}}
\put(120,20){\circle*{3}}
\put(120,40){\circle*{3}}

\multiput(220,0)(20,0){5}{\circle*{3}}
\put(205,0){\textcolor{black}{\line(1,0){110}}}
 \put(190,-3){\textcolor{blue}{$d_1$}}
 \put(320,-3){\textcolor{blue}{$d_2$}}
 
 \put(180,12){\tcr{$d_1+3$}}
 \put(297,43){\tcr{$1$}}

 \put(310,12){\tcr{$d_2+2$}}

\linethickness{1pt}
\put(180,-7){{\line(1,0){30}}}
\put(180,7){{\line(1,0){30}}}
\put(180,-7){{\line(0,1){14}}}
\put(210,-7){{\line(0,1){14}}}

\put(310,-7){{\line(1,0){30}}}
\put(310,7){{\line(1,0){30}}}
\put(310,-7){{\line(0,1){14}}}
\put(340,-7){{\line(0,1){14}}}

\linethickness{.1pt}
\put(220,0){\line(0,1){40}}
\put(220,20){\circle*{3}}
\put(220,40){\circle*{3}}

\put(300,0){\line(0,1){40}}
\put(300,20){\circle*{3}}
\put(300,40){\circle*{3}}

\end{picture}

\begin{picture}(350,70)(0,0)
\multiput(40,0)(20,0){5}{\circle*{3}}
\put(25,0){\textcolor{black}{\line(1,0){110}}}
 \put(10,-3){\textcolor{blue}{$d_1$}}
 \put(140,-3){\textcolor{blue}{$d_2$}}
 \put(0,12){\tcr{$d_1+3$}}
 \put(150,12){\tcr{$d_2$}}
\put(117,43){\tcr{$2$}}
\put(97,03){\tcr{$1$}}

\linethickness{1pt}
\put(0,-7){{\line(1,0){30}}}
\put(0,7){{\line(1,0){30}}}
\put(0,-7){{\line(0,1){14}}}
\put(30,-7){{\line(0,1){14}}}

\put(130,-7){{\line(1,0){30}}}
\put(130,7){{\line(1,0){30}}}
\put(130,-7){{\line(0,1){14}}}
\put(160,-7){{\line(0,1){14}}}

\linethickness{.1pt}
\put(40,0){\line(0,1){40}}
\put(40,20){\circle*{3}}
\put(40,40){\circle*{3}}

\put(120,0){\line(0,1){40}}
\put(120,20){\circle*{3}}
\put(120,40){\circle*{3}}

\multiput(220,0)(20,0){5}{\circle*{3}}
\put(205,0){\textcolor{black}{\line(1,0){110}}}
 \put(190,-3){\textcolor{blue}{$d_1$}}
 \put(320,-3){\textcolor{blue}{$d_2$}}
 
 \put(180,12){\tcr{$d_1+2$}}
 
 \put(217,43){\tcr{$1$}}

 \put(297,43){\tcr{$1$}}

 \put(310,12){\tcr{$d_2+2$}}

\linethickness{1pt}
\put(180,-7){{\line(1,0){30}}}
\put(180,7){{\line(1,0){30}}}
\put(180,-7){{\line(0,1){14}}}
\put(210,-7){{\line(0,1){14}}}

\put(310,-7){{\line(1,0){30}}}
\put(310,7){{\line(1,0){30}}}
\put(310,-7){{\line(0,1){14}}}
\put(340,-7){{\line(0,1){14}}}

\linethickness{.1pt}
\put(220,0){\line(0,1){40}}
\put(220,20){\circle*{3}}
\put(220,40){\circle*{3}}

\put(300,0){\line(0,1){40}}
\put(300,20){\circle*{3}}
\put(300,40){\circle*{3}}

\end{picture}

\begin{picture}(350,70)(0,0)
\multiput(40,0)(20,0){5}{\circle*{3}}
\put(25,0){\textcolor{black}{\line(1,0){110}}}
 \put(10,-3){\textcolor{blue}{$d_1$}}
 \put(140,-3){\textcolor{blue}{$d_2$}}
 \put(0,12){\tcr{$d_1+2$}}
 \put(37,43){\tcr{$1$}}
 \put(150,12){\tcr{$d_2$}}
\put(117,43){\tcr{$2$}}
\put(97,03){\tcr{$1$}}

\linethickness{1pt}
\put(0,-7){{\line(1,0){30}}}
\put(0,7){{\line(1,0){30}}}
\put(0,-7){{\line(0,1){14}}}
\put(30,-7){{\line(0,1){14}}}

\put(130,-7){{\line(1,0){30}}}
\put(130,7){{\line(1,0){30}}}
\put(130,-7){{\line(0,1){14}}}
\put(160,-7){{\line(0,1){14}}}

\linethickness{.1pt}
\put(40,0){\line(0,1){40}}
\put(40,20){\circle*{3}}
\put(40,40){\circle*{3}}

\put(120,0){\line(0,1){40}}
\put(120,20){\circle*{3}}
\put(120,40){\circle*{3}}

\multiput(220,0)(20,0){5}{\circle*{3}}
\put(205,0){\textcolor{black}{\line(1,0){110}}}
 \put(190,-3){\textcolor{blue}{$d_1$}}
 \put(320,-3){\textcolor{blue}{$d_2$}}

 \put(180,12){\tcr{$d_1$}}
 \put(217,43){\tcr{$2$}}
 \put(237,03){\tcr{$1$}}
  \put(277,03){\tcr{$1$}}
 \put(297,43){\tcr{$2$}}
 \put(330,12){\tcr{$d_2$}}

\linethickness{1pt}
\put(180,-7){{\line(1,0){30}}}
\put(180,7){{\line(1,0){30}}}
\put(180,-7){{\line(0,1){14}}}
\put(210,-7){{\line(0,1){14}}}

\put(310,-7){{\line(1,0){30}}}
\put(310,7){{\line(1,0){30}}}
\put(310,-7){{\line(0,1){14}}}
\put(340,-7){{\line(0,1){14}}}

\linethickness{.1pt}
\put(220,0){\line(0,1){40}}
\put(220,20){\circle*{3}}
\put(220,40){\circle*{3}}

\put(300,0){\line(0,1){40}}
\put(300,20){\circle*{3}}
\put(300,40){\circle*{3}}

\end{picture}

\vspace{5mm}
The cases depicted above show the smallest distance between legal 
protrusions that result in diametrical broadcast.  
\end{proof}

We have shown that the only possible diametrical trees must be lobster graphs where the number of limbs is less than half the 
diameter of the tree and the distance between pairs of adjacent limbs or an endpoint satisfies the inequalities in 
Theorem~\ref{thm:MainDiametrical}.

\subsection{Diametrical Grids and Cycles}

The following result is a simple corollary to Theorem \ref{thm:CycleGammab}.
\begin{corollary}\label{corcycle}
The cycle $C_n$ is diametrical if and only if $n=3,4,$ or $5$.
\end{corollary}

\begin{proof}
One can easily verify that $\Gamma_b(C_3) = 1 = \diam{C_3}$ and \[ \Gamma_b(C_5) = \Gamma_b(C_4)= 2 = \diam{C_4} =\diam{C_5}.\]
Thus $C_3,C_4,$ and $C_5$ are diametrical.  To see that no other cycle is diametrical, we recall that the diameter of $C_n$ is 
$\diam{C_n} = \floor{\frac{n}{2}}$ and by applying Theorem \ref{thm:CycleGammab} we find that 
\[ \diam{C_n} =  \floor{\frac{n}{2}} 
\leq \frac{n}{2} \leq n-\frac{n}{2} < n-3 \leq \Gamma_b(C_n). \qedhere \]
\end{proof}

\begin{corollary} \label{cor:justthe2ofus}
Toroidal grids are never diametrical.
\end{corollary}

\begin{proof} 
Without loss of generality assume $3\leq m\leq n$.  Then we have $\diam{\TG{m}{n}}=\floor{m/2}+\floor{n/2}$.  By 
Theorem~\ref{theorem:TGIF} we have
$\Gamma_b(\TG{m}{n})=m\Gamma_b(C_n)$.
We begin by fixing $m=3$ and inducting on $n$.

Base Case: Assume $m=n=3$.  Then
$$\diam{\TG{3}{3}}=\floor{3/2}+\floor{3/2}=2<3=3\Gamma_b(C_3).$$

Inductive Step:  Let $m=3$ and  $n> 3$.   Assume the desired result holds for $n$.
When $n$ is even, this implies the result
$$\diam{\TG{3}{n}}=\floor{3/2}+\floor{n/2}=1+n/2< 3(n-2)$$
holds.  Then for $n+1$ we have
$$\diam{\TG{3}{n+1}}=\floor{3/2}+\floor{(n+1)/2}=1+n/2<3(n-2)=3((n+1)-3)$$
where the third step follows from the inductive hypothesis.
When $n$ is odd,this implies the result
$$\diam{\TG{3}{n}}=\floor{3/2}+\floor{n/2}=1+\floor{n/2}< 3(n-3)$$
holds.  Then for $n+1$ we have
\begin{align*}
\diam{\TG{3}{n+1}}&=\floor{3/2}+\floor{(n+1)/2}\\
&=2+\floor{n/2}\\
&<7+\floor{n/2}\\
&=1+\floor{n/2}+6\\
&<(3n-9)+6\\
&=3n-3=3(n-1)=3((n+1)-2)
\end{align*}
where the fifth step follows from the inductive hypothesis.

The above argument is the base case for induction on $m$.  We assume that
$\diam{\TG{m}{n}}<\Gamma_b(\TG{m}{n})$ for  a fixed $m$ and all $n\geq m$.
If $m$ is even, this implies the result
$$\diam{\TG{m}{n}}=\floor{m/2}+\floor{n/2}=m/2+\floor{n/2}<\begin{cases}m(n-2)\hspace{2mm}\text{ if $n$ is 
even}\\m(n-3)\hspace{2mm}\text{ if $n$ is odd}\end{cases}$$
holds.  In particular, the results holds for a fixed $n\geq m+1$.  Then for $m+1$ we have 
$$\diam{\TG{m}{n}}=m/2+\floor{n/2}<\begin{cases}m(n-2)\hspace{2mm}\text{ if $n$ is even}\\m(n-3)\hspace{2mm}\text{ if $n$ is 
odd}\end{cases}<\begin{cases}(m+1)(n-2)\hspace{2mm}\text{ if $n$ is even}\\(m+1)(n-3)\hspace{2mm}\text{ if $n$ is 
odd}\end{cases}.$$

The case where $m$ is odd is similar.
Hence the graphs $\TG{m}{n}$ are non-diametrical.
\end{proof}

\begin{corollary} \label{cor:griddygridgrid}
The only grid that is diametrical is $G_{2,2} = P_2 \square P_2$.  
\end{corollary}

\begin{proof} We assume that without loss of generality that $1 \leq m \leq n$.
Label the vertices of the grid $G_{m,n} = P_m \square P_n$ by $V=\{v_{1,1},v_{1,2}, \ldots, v_{1,n}, v_{2,1}, 
\ldots v_{2,n}, \ldots , v_{1,n}, \ldots, v_{n,n} \}$ and define a broadcast $f$ on $G_{m,n}$ so that $f(v_{i,1})= n-1$.  
This broadcast is 
minimal, which shows that $\Gamma_b(G_{m,n}) \geq m(n-1)$.  
Also note that the diameter of the grid $G_{m,n}$ is $\diam{G_{m,n}} = m+n-2$.
Therefore, in order for $\diam{P_m\square P_n}=\Gamma_b(P_m\square
P_n)$ we must have
\begin{eqnarray*}
m+n-2 \geq m \cdot (n-1) & \iff & m+n-2=mn-m \\
 & \iff & 2m-mn = 2-n \\
 & \iff & (2-n)m = 2-n \\
& \iff & (n-2)(m-1)=0
\end{eqnarray*}
Hence $G_{m,n}$ is diametrical if and only if $m=1$ and $n \in \mathbb{Z}$, that is, a path $P_n$,
or when $n=2$ and $1 \leq m \leq 2$, that is, the grid $G_{2,2}$.
\end{proof}

\subsection{Diametrical Graphs with Non-diametrical Subgraphs}

In section~\ref{sectionfourpointone}, we are able to classify diametrical trees.  In this section, we consider another obvious type of graph, the cycle.  
By Corollary~\ref{corcycle}, we know that the cycle is non-diametrical for all $n\geq 6$.
Our goal was to take the cycle and modify it by adding paths to vertices on opposite (or near opposite if the cycle is odd) sides of the cycle in an effort to increase the diameter of the graph.  
The hope was that the diameter would increase at a greater pace than the upper broadcast domination number.  We observe this in the example below.  
One can check that although this pattern of modification to create a diametrical graph from a cycle works for a few cycles, it breaks down once we consider $C_{12}$.

It should be noted that the following example shows that there are diametrical graphs 
such that they have non-diametrical subgraphs.  In this case the non-diametrical subgraph is that of $C_6$.

\begin{ex}\label{C6modified}
Note that the cycle $C_6$ is non-diametrical while the graph below, $G'$, is diametrical.  By adding leaves to opposite sides of 
the 
cycle, we increase the diameter of the graph from 3 to 5.  The upper broadcast domination number of the cycle $C_6$ is 4 while 
the 
upper broadcast domination number of $G'$ is 5.  We  note here that the spanning trees of $G'$ are also diametrical.
\end{ex}

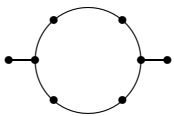
\begin{center}
\begin{picture}(100,50)(0,30)

\put(50,50){\circle{50}}
\put(30,50){\circle*{3}}
\put(70,50){\circle*{3}}
\put(30,50){\line(-1,0){10}}
\put(20,50){\circle*{3}}
\put(70,50){\line(1,0){10}}
\put(80,50){\circle*{3}}
\put(63,65){\circle*{3}}
\put(37,65){\circle*{3}}
\put(37,35){\circle*{3}}
\put(63,35){\circle*{3}}
\end{picture}
\end{center}

\section{Open Questions}\label{open}

We conclude this paper with a list of open questions raised by our results or 
restated from references. Many of these questions may serve as good primers for research projects
with master's and undergraduate students.

\begin{question}
Characterize classes of graphs (other than trees) for which $\Gamma_b(G)=\diam{G}$.
\end{question}

\begin{question} Consider two  invariants which are incomparable when considering arbitrary graphs; see ~\cite{DEHHH05}  for a list of  invariants.  Do there 
exist classes of graphs 
for 
which the invariants are comparable? If so what is the comparison?
\end{question}

\begin{question}\cite{DEHHH05}
If only limited broadcast powers are allowed for a graph, that is, a $k$-limited broadcast domination number $\gamma_{kb}(G)$, 
what can be said about the invariant?  
\end{question}

\begin{question}\cite{DEHHH05}
What can you say about the class of minimum cost dominating broadcasts, where the number of broadcast vertices is a minimum (or 
maximum)?
\end{question}

\begin{question}
What can be said about the upper domination number and upper broadcast domination number of the product or strong  product of 
cycles?
\end{question}

\begin{question}
Classify the diametrical graphs $G$ for which $G$ is the cartesian product of graphs.
\end{question}

\section{Acknowledgments}
We would like to thank the Society for the Advancement of Chicanos and Native Americans in Science 
(SACNAS), Shannon Talbott and 
Pamela Harris for creating the SACNAS Collaborative Minigrants which provided travel support for this project through  the NSA award (H98230-15-1-0091) and NSF award DMS-1545136.
We also thank Mohamed Omar for useful conversations.

\end{document}